\documentclass[a4paper]{article} % A4 stock; 'article' is standard for papers/notes

\usepackage[
a4paper,                          % A4 stock
margin=28mm,                      % Uniform text margins
marginparwidth=22mm,              % Wide enough for todonotes
includehead, includefoot          % Count header/footer in text area (optional)
]{geometry}

\usepackage{xcolor}                 % Colors for text, rules, links, etc.

\usepackage{multicol}               % \begin{multicols}{2} ... \end{multicols}

\usepackage{longtable,booktabs}     % Long tables + high-quality horizontal rules

\usepackage{graphicx}               % \includegraphics, \graphicspath
\graphicspath{{figures/}}           % (Optional) default path for figures

\usepackage{amsmath}                % Math environments (align, gather, etc.)

\renewcommand{\Re}{\operatorname{Re}}

\renewcommand{\Im}{\operatorname{Im}}

\usepackage{amssymb}                % Math symbols (\mathbb, \mathfrak, etc.)
\usepackage{amsthm}                 % \newtheorem environments
\usepackage{mathtools}              % (Optional) amsmath enhancements (\coloneqq, \DeclarePairedDelimiter, ...)

\usepackage{mathrsfs}               % Provides \mathscr (formal script letters)

	\usepackage{silence}                % Selectively suppress warnings
	\usepackage[T1]{fontenc}                 % Proper glyphs & hyphenation in PDF
	\usepackage[utf8]{inputenc}              % UTF-8 source for pdfLaTeX
	\usepackage{lmodern}                     % Latin Modern fonts (vector)
	\usepackage[nopatch=footnote]{microtype} % Improves justification; no footnote patch (avoid conflicts)
	
	\usepackage{ragged2e}
	
	\usepackage[normalem]{ulem}         % \sout{}, but keep \emph behavior normal
	
	\usepackage{titlesec}               % Custom section formats
	\usepackage{setspace}               % \onehalfspacing, \doublespacing etc.
	\usepackage{indentfirst}            % Indent first paragraph of each section
	
	\usepackage[inline]{enumitem}
	
	\newcommand{\papertitle}{On the Approximation of the Unitary Operator Group Associated with a Rotation Matrix and Its Applications to Abstract Hyperbolic Equations}
	
	\newcommand{\papershorttitle}{Unitary Group Approximation and Abstract Hyperbolic Equations}
	
	\newcommand{\authorRogava}{Jemal Rogava}
	\newcommand{\authorVashakidze}{Zurab Vashakidze}
	
	\newcommand{\orcidRogava}{0000-0001-9460-4283}
	
	\newcommand{\orcidVashakidze}{0000-0001-8736-6213}
	
	\newcommand{\paperauthors}{\authorRogava; \authorVashakidze}
	
	\newcommand{\runauthors}{\authorRogava{} \textbullet\ \authorVashakidze} % Authors in header
	\newcommand{\runningtitle}{\papershorttitle} % Adjust to a shorter variant if needed
	
	\usepackage{fancyhdr}               % Custom header/footer management
	\AtBeginDocument{\thispagestyle{empty}} % Also set in titlepage below; redundancy is harmless
	
	\fancypagestyle{plain}{%
		\fancyhf{}
		\fancyhead[C]{%
			\ifodd\value{page}%
			\small \runauthors%
			\else%
			\small \runningtitle%
			\fi%
		}
		\fancyfoot[C]{\thepage}

	}
	
	\definecolor{linkColor}{RGB}{155,0,119} % Custom link color (accessible purple)
	\usepackage[
	pdfstartview  = {FitH},          % Open PDF with page width fitting window
	pdftitle      = {\papertitle},   % PDF metadata: title (from macro)
	pdfauthor     = {\paperauthors}, % PDF metadata: author(s) (from macro)
	pdfusetitle,                     % Also use \title and \author for PDF metadata
	colorlinks    = true,            % Colorize links instead of boxed links
	linkcolor     = linkColor,       % Section/cross-ref link color
	citecolor     = linkColor,       % Citation link color
	urlcolor      = linkColor,       % URL link color
	hypertexnames = false            % Avoid duplicate anchors (safer TOC/refs/cross-refs)
	]{hyperref}                         % Load hyperref once with options
	\usepackage{orcidlink} % Official ORCID iD mark, works with pdfLaTeX/XeLaTeX/LuaLaTeX
	
	\usepackage[nameinlink]{cleveref}   % Smart cross-references; auto “Fig.”, “Sec.”, etc.
	\crefname{equation}{eq.}{eqs.}      % "eq. (1)" / "eqs. (1) and (2)"
	\Crefname{equation}{Eq.}{Eqs.}      % "Eq. (1)" at sentence start
	\crefname{theorem}{theorem}{theorems}
	\Crefname{theorem}{Theorem}{Theorems}
	\crefname{figure}{figure}{figures}
	\Crefname{figure}{Figure}{Figures}
	\crefname{table}{table}{tables}
	\Crefname{table}{Table}{Tables}
	
	\theoremstyle{plain}
	
	\newtheorem{theorem}{Theorem}[section]      % Title: Theorem (numbered by section)
	
	\newtheorem*{theorem*}{Theorem}             % Title: Theorem* (unnumbered)
	
	\newtheorem{lemma}[theorem]{Lemma}          % Title: Lemma (shares counter)
	
	\newtheorem*{lemma*}{Lemma}                 % Title: Lemma* (unnumbered)
	
	\newtheorem*{corollary*}{Corollary}         % Title: Corollary* (unnumbered)
	
	\theoremstyle{definition}
	
	\newtheorem*{definition*}{Definition}        % Title: Definition* (unnumbered)
	
	\newtheoremstyle{boldremark} % Name
	{3pt}                      % Space above: vertical skip before the environment
	{3pt}                      % Space below: vertical skip after the environment
	{\normalfont}              % Body font: upright/roman text (not italic)
	{}                         % Indent amount: empty = no indentation
	{\bfseries}                % Head font: bold heading (e.g., "Remark 1.2.")
	{.}                        % Punctuation after head: prints a period after heading
	{.5em}                     % Space after head: spacing between head and body text
	{}                         % Head specification: empty = default layout
	
	\theoremstyle{boldremark}
	
	\newtheorem{remark}[theorem]{Remark}        % Title: Remark
	
	\newtheorem*{remark*}{Remark}               % Title: Remark* (unnumbered)
	
	\theoremstyle{definition}
	
	\newtheorem*{test*}{Test}                   % Title: Test* (unnumbered)
	
	\newcommand{\refitem}[2]{%
		\hyperref[#2]{#1~\ref*{#2}}% \ref* prints the bare number; \hyperref makes it a single link
	}
	
	\newcommand{\figsubref}[2]{%
		\hyperref[#2]{Figure~\ref*{#1}(\subref*{#2})}% one clickable "Figure <main>(<sub>)"
	}
	
	\newcommand{\lemref}[1]{%
		\refitem{Lemma}{#1}%
	}
	
	\newcommand{\thmref}[1]{%
		\refitem{Theorem}{#1}%
	}
	
	\newcommand{\remref}[1]{%
		\refitem{Remark}{#1}%
	}
	
	\makeatletter
	\renewenvironment{proof}[1][\proofname]{%    % Title: proof environment (redefinition)
		\par\pushQED{\qed}\normalfont%           % Push end-of-proof symbol; normal font
		\topsep6\p@\@plus6\p@\relax%             % Vertical spacing before proof
		\trivlist\item[\hskip\labelsep\bfseries#1\@addpunct{.}]% % Bold "Proof." label with punctuation
		\ignorespaces%                           % Ignore spaces after the label
	}{%
		\popQED\endtrivlist\@endpefalse%         % Pop QED symbol; end list; reset end-paragraph flag
	}
	\makeatother
	
	\long\def\d{\mathrm{d}}            % Upright differential 'd' for dx, dy, etc.
	\long\def\bigO{\mathcal{O}}        % Big-O notation
	\long\def\A{\mathcal{A}}           % Calligraphic A
	\long\def\B{\mathcal{B}}           % Calligraphic B
	\long\def\I{\mathcal{I}}           % Calligraphic I
	\long\def\F{\mathcal{F}}           % Calligraphic F
	\long\def\Hilb{\mathcal{H}}        % Calligraphic H (Hilbert space)
	\long\def\ii{\mathrm{i}}           % Upright imaginary unit
	\long\def\Real{\mathbb{R}}		   % Set of real numbers
	\long\def\W{\mathcal{W}}           % Calligraphic W 
	
	\long\def\ie{\textit{i.e.}}      % Latin abbreviation “that is”; italicized. (No auto-space—add a trailing space in text.)
	\long\def\eg{\textit{e.g.}}      % Latin abbreviation “for example”; italicized. (Add space after use.)
	\long\def\cf{\textit{cf.}}       % Latin “compare”; italicized. (Add space after use.)
	\titleformat{\section}{\large\bfseries}{\thesection}{1em}{}            % Title: \section format (bold, larger)
	\titleformat{\subsection}{\normalsize\bfseries}{\thesubsection}{1em}{} % Title: \subsection format (bold)
	
	\usepackage[%
				backend=biber,%	
				style=numeric-comp,   % Compressed numeric citations [1–3]
				safeinputenc=true     % Preserve TeX accent commands; avoid Unicode combining character issues with pdflatex
			   ]{biblatex}
	
	\DeclareUnicodeCharacter{0306}{\protect\u{}{}}
	\numberwithin{equation}{section} % Equations: (section.eqnum)
	\numberwithin{figure}{section}   % Figures: (section.fignum)
	\numberwithin{table}{section}    % Tables: (section.tabnum)
	
	\usepackage[most]{tcolorbox} % Load tcolorbox with common libraries (skins, breakable, theorems)
	\tcbset{%
		myabstract/.style={%
			colback=gray!5!white,    % Light background
			colframe=black!20,       % Frame color
			fonttitle=\bfseries,     % Bold title
			coltitle=black,          % Title text color
			boxrule=0.6pt,           % Border thickness
			arc=2mm,                 % Rounded corners
			enhanced,                % Enable advanced options
			before skip=10pt,        % Space before the box
			after skip=20pt,         % Space after the box
			breakable                % Allow page breaks if the abstract is long
		}%
	}
	\renewenvironment{abstract}%      % Title: abstract environment (redefinition)
	{\begin{tcolorbox}[myabstract,title=Abstract]}%  % Begin styled abstract box
		{\end{tcolorbox}}%                           % End styled abstract box
	
	\title{\papertitle}                 % For PDF metadata/bookmarks
	\author{\runauthors}                % For PDF metadata/bookmarks
	\date{}                             % Empty date suppresses date printing under the title
	
	\newcommand{\titleseparator}{\par\noindent{\color{black}\rule{\linewidth}{0.8pt}}\par}
	
	\usepackage{todonotes}
	
	\allowdisplaybreaks  % Allow multi-line displayed equations to break across pages
	
\begin{document}
		
		\sloppy  % Relax line-breaking rules globally (looser spacing, fewer overfull boxes)
		
		% ======================================================================
		%  CUSTOM TITLE PAGE
		%  Title: Title Page (manual)
		%  Scope: Title between rules; centered author blocks; abstract with \vfill
		%  Notes: All body text is black by default; links use 'linkColor'.
		% ======================================================================
		
		\begin{titlepage}
			%==============================================================
			% Title Page
			% - Fully centered, with flexible vertical spacing controlled by \vfill.
			% - Uses three \vfill "springs" to distribute whitespace between:
			%   (1) title block, (2) author blocks, (3) abstract/keywords block.
			%==============================================================
			
			\thispagestyle{empty} % Suppress headers/footers on the title page
			
			\begin{center}
				%----------------------------------------------------------
				% Vertical balancing (top spring)
				% - \vfill expands to occupy available space, helping center the content
				%   vertically on the page (as opposed to a fixed \vspace*).
				%----------------------------------------------------------
				\vfill
				
				%----------------------------------------------------------
				% Title block
				% - \titleseparator: your custom horizontal rule used above/below the title.
				% - \vspace{...}: fixed micro-spacing around the title for consistent rhythm.
				%----------------------------------------------------------
				\titleseparator
				\vspace{0.9em}
				{\LARGE\bfseries \papertitle\par}
				\vspace{0.7em}
				\titleseparator
				
				%----------------------------------------------------------
				% Vertical balancing (middle spring)
				% - Separates the title from the author section while keeping the
				%   overall layout vertically balanced.
				%----------------------------------------------------------
				\vfill
				
				%==========================================================
				% Author 1 block
				% - Name + ORCID: emphasized line (large + bold).
				% - Affiliation: structured line breaks for readable hierarchy.
				% - Email: monospaced + clickable mailto link.
				%==========================================================
				{\large\textbf{\authorRogava}\,\orcidlink{\orcidRogava}\par}
				\vspace{0.4em}
				{\small
					Faculty of Exact and Natural Sciences,\\
					Ivane Javakhishvili Tbilisi State University (TSU),\\
					Ilia Vekua Institute of Applied Mathematics (VIAM),\\
					11 University Street, 0186 Tbilisi, Georgia\par}
				\vspace{0.3em}
				{\small
					\href{mailto:jemal.rogava@tsu.ge}{\texttt{jemal.rogava@tsu.ge}}\par}
				
				%----------------------------------------------------------
				% Vertical balancing between authors (spring)
				% - \vfill gives adaptive separation: it grows/shrinks depending on
				%   the remaining page height, keeping blocks visually balanced.
				%----------------------------------------------------------
				\vfill
				
				%==========================================================
				% Author 2 block (multiple affiliations + multiple emails)
				% - Uses \\[0.4em] to insert a small visual break between institutions.
				% - Emails separated by \textbar for compact presentation.
				%==========================================================
				{\large\textbf{\authorVashakidze}\footnote{Corresponding author: {\authorVashakidze}. E-mail addresses: \href{mailto:z.vashakidze@ug.edu.ge}{\texttt{z.vashakidze@ug.edu.ge}}; \href{mailto:zurab.vashakidze@tsu.ge}{\texttt{zurab.vashakidze@tsu.ge}}}\,\orcidlink{\orcidVashakidze}\par}
				\vspace{0.4em}
				{\small
					School of Science and Technology, The University of Georgia (UG),\\
					77 Merab Kostava Street, 0171 Tbilisi, Georgia\\[0.4em]
					Ilia Vekua Institute of Applied Mathematics (VIAM),\\
					Ivane Javakhishvili Tbilisi State University (TSU),\\
					11 University Street, 0186 Tbilisi, Georgia\par}
				\vspace{0.3em}
				{\small
					\href{mailto:z.vashakidze@ug.edu.ge}{\texttt{z.vashakidze@ug.edu.ge}} \textbar\
					\href{mailto:zurab.vashakidze@tsu.ge}{\texttt{zurab.vashakidze@tsu.ge}}\par}
				
				%----------------------------------------------------------
				% Vertical balancing (bottom spring)
				% - Pushes the abstract/keywords block toward the lower part of the page
				%   while still allowing elastic whitespace above it.
				%----------------------------------------------------------
				\vfill
				
				%----------------------------------------------------------
				% Abstract + Keywords block (width-constrained)
				% - minipage at 0.9\textwidth improves line length and readability.
				% - Inner \vfill separates abstract and keywords when extra space exists.
				%----------------------------------------------------------
				\begin{minipage}{0.9\textwidth}
					
					% Abstract environment (styled by your class/package)
					\begin{abstract}
						The solution of the Cauchy problem for homogeneous abstract hyperbolic equations, together with its derivative, admits a vector representation in terms of a unitary operator group associated with a rotation matrix. A rational approximation of this unitary group is constructed and shown to possess optimal fourth-order convergence. The order of convergence is determined in accordance with the smoothness scale. Based on this rational approximation, a two-layer semi-discrete scheme is constructed for the approximate solution of Cauchy problems for nonhomogeneous abstract hyperbolic equations in both the linear and semilinear settings. The convergence properties of the scheme are examined in relation to the regularity of the solution.
					\end{abstract}
					
					\vfill
					
					% Keywords line: compact, journal-friendly
					
					\vspace{6pt}
					
					% MSC2020 classification line: compact, journal-friendly formatting
				\end{minipage}
				
			\end{center}
		\end{titlepage}
		
		% ======================================================================
		%  METHOD: New Page
		%  Title: Hard Page Break
		%  Role: Start main content on a fresh page; stronger than \pagebreak.
		% ======================================================================
		\newpage
		
		\section{Introduction}
		
		It is well known that the solution of the Cauchy problem for abstract hyperbolic equations is given by the operator functions sine and cosine. This fact allows the unknown function and its derivative to be defined vectorially by an operator matrix dependent on the time variable, which forms a unitary operator group. We observe that this unitary group may be naturally interpreted as being associated with a rotation matrix.
		
		In the paper, a two-layer vector scheme is constructed using a high-order of accuracy rational approximation for the unitary operator group. Using this scheme, we can find both the value of the unknown function and its derivative at each time layer. The stability of the constructed scheme is investigated; in particular, it is proved that the norm of the operator matrix of transition does not exceed one, which ensures the stability of the considered scheme in any finite time interval. It is proven that the constructed rational approximation, which is in fact an operator analogue of scalar Padé approximation, gives the value of the unknown function and its derivative at any time step with fourth-order accuracy. Observe that Padé rational approximation is also employed in \cite{RogavaOperSemigroup2022} for semigroups of operators.
		
		It should be noted that the fourth-order convergence of the proposed scheme is optimal under the prescribed smoothness assumptions and cannot be improved by imposing additional regularity. Moreover, we establish the dependence of the convergence order on the smoothness of the solution, starting from the case in which the solution satisfies the minimal smoothness requirements.
		
		Cosine operator functions play an important role in representing the solution of the Cauchy problem for an abstract hyperbolic equation. The definition of the cosine operator function in a general Banach space is addressed in the work of Sova \cite{Sova1966}. The main result of this article is a characterization of the operators that can generate cosine operator functions. Hoppe \cite{Hoppe1982} considered the approximation of cosine operator functions in terms of their infinitesimal generators and discussed the consequences of these results for the approximate solution of the initial-value problem associated with an abstract homogeneous hyperbolic equation. The article by Goldstein \cite{Goldstein1974} also addresses the approximation of cosine operator functions. The author proved an analogue of the Trotter-Neveu-Kato semigroup approximation theorem for cosine functions.
		
		The existence, uniqueness, continuous dependence on the initial data, and regularity of solutions to the Cauchy problem for semilinear abstract hyperbolic equations in Banach spaces were investigated by Travis and Webb in \cite{TravisWebb1978}. It should be noted that, in the equation treated in that work, the nonlinear term is assumed to satisfy a Lipschitz condition with respect to both the unknown function and its derivative.
		
		We note that important contributions to the construction and analysis of numerical schemes for the approximate solution of Cauchy problems associated with hyperbolic equations were made by Baker \cite{Baker1976}, Baker and Bramble \cite{BakerBramble1979}, Baker, Dougalis, and Serbin \cite{BakerDougalisSerbin1980,BakerDougalisSerbin1979}, Bales \cite{Bales1993}, Ka\v{c}ur \cite{Kacur1984}, Pultar \cite{Pultar1984}, and Sobolevski\u{\i} and \v{C}ebotareva \cite{SobolevskiiChebotareva1977}, among others.
		
		In \cite{BakerDougalisSerbin1980} and \cite{BakerDougalisSerbin1979}, the authors proposed an approach that differs from the standard one. Starting from an analytic representation of the solution to a homogeneous abstract hyperbolic equation, they constructed an exact three-level semidiscrete scheme on a uniform partition of the time interval. The transition operator associated with this scheme is given by the corresponding cosine operator function. The recurrence relation thereby obtained yields semidiscrete schemes of arbitrary order of accuracy, whose transition operators are rational approximations of the cosine operator. A related approach was employed by Serbin in \cite{Serbin1979} for second-order systems of differential equations.
		
		More recent developments in the time discretization of second-order evolution equations include rational approximations of cosine and trigonometric operator functions, Runge–Kutta–Nystr\"{o}m methods, and trigonometric and exponential integrators. In \cite{AlonsoMalloCanoMoreta2009}, Alonso-Mallo, Cano, and Moreta studied the stability of rational approximations of cosine functions for abstract second-order problems in Hilbert spaces, while Grimm and Hochbruck \cite{GrimmHochbruck2008} investigated rational approximations of trigonometric operator functions arising in the numerical solution of wave equations. Multistep cosine methods that integrate the stiff linear part of second-order partial differential equations exactly were constructed and analyzed by Cano and Moreta in \cite{CanoMoreta2010}. Trigonometric and exponential time integrators have also been studied extensively for semilinear wave equations. In particular, Gauckler \cite{Gauckler2015} established convergence results for trigonometric integrators under finite-energy regularity assumptions, whereas Buchholz, D\"{o}rich, and Hochbruck \cite{BuchholzDoerichHochbruck2021} developed exponential time-integration schemes for abstract semilinear wave equations with low-regularity solutions.
		
		We recall a standard approach to the treatment of second-order evolution problems. Introducing an auxiliary unknown, one may reformulate the problem as a first-order evolution equation governed by a matrix operator. Under suitable assumptions on the operator appearing in the original problem, the solution of the resulting first-order equation admits a representation in terms of an operator semigroup. The theory of operator semigroups is well developed and constitutes a fundamental tool in the analysis of evolution problems; see, for instance, Arendt \cite{Arendt2004}, Arendt, Vogt, and Voigt \cite{ArendtVogtVoigt2026}, and Hille and Phillips \cite{HillePhillips1974}. We further refer to the monograph of Zagrebnov, Neidhardt, and Ichinose \cite{ZagrebnovNeidhardtIchinose2024}, where the theory of semigroup approximation is systematically developed in both Hilbert and Banach spaces. Finally, we note that \cite{GegechkoriRogavaTsiklauri2002} and \cite{GegechkoriRogavaTsiklauri2004} are devoted to the construction of high-order splitting schemes for evolution problems based on semigroup approximation.
		
		At the end of the introduction, we would like to note that the proposed approximation of the unitary operator group associated with the rotation matrix yields approximate values of the solution of the Cauchy problem for a homogeneous abstract hyperbolic equation, as well as of its derivative, at uniformly spaced nodes of the time interval. In the case of a nonhomogeneous linear equation, retaining the same approximation for the corresponding homogeneous problem and approximating the integral term by a quadrature formula chosen in accordance with the regularity of the solution leads to a combined scheme that approximates both the solution of the original problem and its derivative at the nodal points. The order of the error with respect to the time step is then determined by the regularity of the solution. The developed approach is extended to the general semilinear setting.
		
		The remainder of the paper is organized as follows. \refitem{Section}{sec:unit_group_rot_matrix} develops the unitary operator group associated with the rotation matrix. \refitem{Section}{sec:approx_unit_group} constructs a stable rational approximation of this group and establishes convergence results reflecting the regularity of the data. Finally, \refitem{Section}{sec:approx_sol_abstr_hyper_eq} applies the resulting approximation framework to a class of linear and semilinear abstract hyperbolic equations.
		
		\section{The Unitary Operator Group Associated with a Rotation Matrix}\label{sec:unit_group_rot_matrix}
		
		Let $\Hilb$ be a Hilbert space equipped with inner product $\left( \cdot,\cdot \right)$ and associated norm $\left\Vert \cdot \right\Vert$. We define the Cartesian product space $\Hilb \times \Hilb$ as the set of all ordered pairs with components in $\Hilb$. Each element of this space is therefore a vector of the form $\left\{ u,v \right\}$, where $u \in \Hilb$ and $v \in \Hilb$. The inner product and the corresponding norm in the Cartesian product space are defined as follows:
		\begin{equation*}
			\left( \left( \left\{ u_1,v_1 \right\},\left\{ u_2,v_2 \right\} \right) \right) = \left( u_1,u_2 \right) + \left( v_1,v_2 \right) \quad \text{and} \quad \left\Vert \left\{ u,v \right\} \right\Vert_{\Hilb \times \Hilb} = \left( {\left\Vert u \right\Vert}^2 + {\left\Vert v \right\Vert}^2 \right)^{\frac{1}{2}}\,.
		\end{equation*}
		
		We define the following operator matrix acting on the product space $\Hilb \times \Hilb$:
		\begin{equation*}
			U \left( t \right) =
			\left(
			\begin{array}{rr}
				\cos \left( t \A^{1 / 2} \right) & \sin \left( t \A^{1 / 2} \right) \\
				-\sin \left( t \A^{1 / 2} \right) & \cos \left( t \A^{1 / 2} \right)
			\end{array}
			\right)\,.
		\end{equation*}
		
		Here, $\A : D \left( \A \right) \subset \Hilb \to \Hilb$ is a densely defined, self-adjoint, positive definite operator (possibly unbounded). The operator-valued functions $\cos \left( t \A^{1 / 2} \right)$ and $\sin \left( t \A^{1 / 2} \right)$ are defined by means of the operator analogue of the classical Euler formulas. More precisely,
		\begin{equation}\label{eq:euler_formulas}
			\cos \left( t \A^{1 / 2} \right) = \frac{1}{2} \left( e^{\ii t \sqrt{\A}} + e^{-\ii t \sqrt{\A}} \right)\,, \quad \sin \left( t \A^{1 / 2} \right) = \frac{1}{2 \ii} \left( e^{\ii t \sqrt{\A}} - e^{-\ii t \sqrt{\A}} \right)\,,
		\end{equation}
		where $\left\{ e^{\pm \ii t \sqrt{\A}} \right\}_{t \in \Real}$ is the unitary group generated by the operator $\mp \ii \A^{1 / 2}$ (\cf{} Stone \cite{Stone1932}, {\bfseries Chapter IX} in Yosida \cite{Yosida1980}).
		
		It is readily verified that
		\begin{equation*}
			\left\Vert \left( \I \pm \lambda \ii \A^{1 / 2} \right)^{-1} \right\Vert \leq 1\,, \quad \text{for all} \quad \lambda \in \Real\,,
		\end{equation*}
		where $\I$ denotes the identity operator on the Hilbert space $\Hilb$.
		
		Consequently, for every $\varphi \in \Hilb$, the following limit exists:
		\begin{equation*}
			\lim_{n \to \infty} \left( \I \pm \frac{t}{n} \ii \A^{1 / 2} \right)^{-n} \varphi\,.
		\end{equation*}
		This limit is denoted by $e^{\mp \ii t \sqrt{\A}} \varphi$ and defines an operator family possessing the characteristic properties of the exponential function (see Kato \cite{Kato1980}, {\bfseries Chapter IX}).
		
		It should be observed that we consider the case when the argument of the operator functions sine and cosine contains the square root of a self-adjoint, positive definite operator, since the solution of the abstract hyperbolic equation is given by these operator functions. Our ultimate objective is to approximate the solution of the abstract hyperbolic equation by approximating the unitary operator group.
		
		We can easily show that the adjoint of the operator $U \left( t \right)$ has the following form:
		\begin{equation*}
			U \left( t \right)^{\ast} =
			\left(
			\begin{array}{rr}
				\cos \left( t \A^{1 / 2} \right) & -\sin \left( t \A^{1 / 2} \right) \\
				\sin \left( t \A^{1 / 2} \right) & \cos \left( t \A^{1 / 2} \right)
			\end{array}
			\right)\,,
		\end{equation*}
		or, equivalently, $U \left( t \right)^{\ast} = U \left( -t \right)$. Furthermore, it is evident that $U \left( t \right) U \left( t \right)^{\ast} = I$, where $I$ is the identity operator on $\Hilb \times \Hilb$,
		\begin{equation*}
			I =
			\left(
			\begin{array}{rr}
				\I & 0 \\
				0 & \I
			\end{array}
			\right)\,.
		\end{equation*}
		
		From the formulas given in \eqref{eq:euler_formulas}, it follows that
		\begin{equation*}
			U \left( t \right) U \left( s \right) = U \left( t + s \right) \quad \text{for all} \quad t,s \in \Real\,.
		\end{equation*}
		
		Using the formulas in \eqref{eq:euler_formulas}, one obtains
		\begin{equation*}
			U \left( t \right) \varphi \to U \left( t_0 \right) \varphi \quad \text{for every} \quad \varphi \in \Hilb \times \Hilb \quad \text{as} \quad t \to t_0\,.
		\end{equation*}
		
		Therefore, the family $\left\{ U \left( t \right) \right\}_{t \in \Real}$ of bounded linear operators on $\Hilb \times \Hilb$ forms a strongly continuous one-parameter unitary group. That is, the following conditions hold:
		\begin{enumerate}[label=(\roman*)]
			\item $U \left( t + s \right) = U \left( t \right) U \left( s \right)$, for all $t,s \in \Real$.
			\item $U \left( 0 \right) = I$ (the identity operator on $\Hilb \times \Hilb$).
			\item For every $t \in \Real$, the operator $U \left( t \right)$ is unitary; that is, $U \left( t \right)^{\ast} = U \left( t \right)^{-1}$.
			\item For every $\varphi \in \Hilb \times \Hilb$, the mapping $t \mapsto U \left( t \right) \varphi$ is continuous with respect to the strong operator topology.
		\end{enumerate}
		
		Observe that, upon replacing the argument $t \A^{1 / 2}$ in the matrix operator $U \left( t \right)$ by the scalar angle $\theta$, one obtains the standard rotation matrix in the plane. Consequently, it is natural to regard the corresponding unitary group $U \left( t \right)$ as being associated with the rotation matrix.
		
		\section{Approximation of a Unitary Operator Group Associated with a Rotation Matrix}\label{sec:approx_unit_group}
		\subsection{Rational Approximation of the Sine and Cosine Operator Functions}
		
		The main idea of this work is to approximate the operator functions $\cos \left( t \A^{1 / 2} \right)$ and $\sin \left( t \A^{1 / 2} \right)$ by rational operator functions $\W_1 \left( t \right)$ and $\W_2 \left( t \right)$, respectively, in such a way that the norm of the sum of their squares does not exceed one.
		
		Consider the following operator matrix acting on $\Hilb \times \Hilb$:
		\begin{equation*}
			V \left( t \right) =
			\left(
			\begin{array}{rr}
				\W_1 \left( t \right) & \W_2 \left( t \right) \\
				-\W_2 \left( t \right) & \W_1 \left( t \right)
			\end{array}
			\right)\,.
		\end{equation*}
		\begin{lemma}\label{prop:lemma1}
			If the condition $\left\Vert \W_1^2 \left( t \right) + \W_2^2 \left( t \right) \right\Vert \leq 1$ holds, then $\left\Vert V \left( t \right) \right\Vert_{\Hilb \times \Hilb} \leq 1$.
		\end{lemma}
		\begin{proof}
			By straightforward algebraic manipulation, one obtains
			\begin{align*}
				\left\Vert V \left( t \right) \left\{ u_1,u_2 \right\} \right\Vert_{\Hilb \times \Hilb}^2 &= \left( \left(  V \left( t \right) \left\{ u_1,u_2 \right\}, V \left( t \right) \left\{ u_1,u_2 \right\} \right) \right) \\
				&= \left( \left(  V \left( t \right)^{\ast} V \left( t \right) \left\{ u_1,u_2 \right\},\left\{ u_1,u_2 \right\} \right) \right) \\
				&= \left( \left(  \left\{ \left[ \W_1^2 \left( t \right) + \W_2^2 \left( t \right) \right] u_1,\left[ \W_1^2 \left( t \right) + \W_2^2 \left( t \right) \right] u_2 \right\},\left\{ u_1,u_2 \right\} \right) \right) \\
				&= \left( \left[ \W_1^2 \left( t \right) + \W_2^2 \left( t \right) \right] u_1,u_1 \right) + \left( \left[ \W_1^2 \left( t \right) + \W_2^2 \left( t \right) \right] u_2,u_2 \right) \\
				&\leq \left\Vert \left[ \W_1^2 \left( t \right) + \W_2^2 \left( t \right) \right] u_1 \right\Vert \left\Vert u_1 \right\Vert + \left\Vert \left[ \W_1^2 \left( t \right) + \W_2^2 \left( t \right) \right] u_2 \right\Vert \left\Vert u_2 \right\Vert \\
				&\leq \left\Vert \W_1^2 \left( t \right) + \W_2^2 \left( t \right) \right\Vert \left( \left\Vert u_1 \right\Vert^2 + \left\Vert u_2 \right\Vert^2 \right) = \left\Vert \W_1^2 \left( t \right) + \W_2^2 \left( t \right) \right\Vert \left\Vert \left\{ u_1,u_2 \right\} \right\Vert_{\Hilb \times \Hilb}^2\,.
			\end{align*}
			Consequently,
			\begin{equation*}
				\left\Vert V \left( t \right) \left\{ u_1,u_2 \right\} \right\Vert_{\Hilb \times \Hilb} \leq \sqrt{\left\Vert \W_1^2 \left( t \right) + \W_2^2 \left( t \right) \right\Vert} \leq 1\,.
			\end{equation*}
		\end{proof}
		
		The next result concerns the approximation of the operator functions $\cos \left( t \A^{1 / 2} \right)$ and $\sin \left( t \A^{1 / 2} \right)$ by rational operator functions in a neighborhood of the point $t = 0$.
		
		\begin{lemma}\label{prop:lemma2}
			Let the rational operator functions $\W_1 \left( t \right)$ and $\W_2 \left( t \right)$ be defined by
			\begin{align*}
				\W_1 \left( t \right) &= \left( \I + \lambda t^2 \A \right)^{-1} \left( \I + \bar{\lambda} t^2 \A \right)^{-1}\,, \quad \lambda = \frac{1}{4} + \ii \frac{\sqrt{21}}{12}\,, \\
				\W_2 \left( t \right) &= t \A^{1 / 2} \left( \I + \frac{1}{6} t^2 \A \right)^{-1} \left( \I + \lambda_0 t^2 \A \right)^{-1} \left( \I + \bar{\lambda}_0 t^2 \A \right)^{-1}\,, \quad \lambda_0 = \ii \frac{\sqrt{15}}{12}\,.
			\end{align*}
			Then the following estimates hold:
			\begin{align}
				\left\Vert \left( \cos \left( t \A^{1 / 2} \right) - \W_1 \left( t \right) \right) u \right\Vert &\leq \frac{31}{120} \left\vert t \right\vert^5 \left\Vert \A^{5 / 2} u \right\Vert\,, \quad u \in D \left( \A^{5 / 2} \right)\,,\label{eq:cos_approx} \\
				\left\Vert \left( \sin \left( t \A^{1 / 2} \right) - \W_2 \left( t \right) \right) u \right\Vert &\leq \frac{101}{720} \left\vert t \right\vert^5 \left\Vert \A^{5 / 2} u \right\Vert\,, \quad u \in D \left( \A^{5 / 2} \right)\,.\label{eq:sin_approx}
			\end{align}
			Moreover, the operator functions $\W_1 \left( t \right)$ and $\W_2 \left( t \right)$ satisfy $\left\Vert \W_1^2 \left( t \right) + \W_2^2 \left( t \right) \right\Vert \leq 1$ for all $t \in \Real$.
		\end{lemma}
		\begin{proof}
			Using the expansion
			\begin{equation*}
				\left( \I + \tau \A \right)^{-1} = \I + \sum_{i = 1}^{k - 1} \left( -1 \right)^i \tau^i \A^i + \left( -1 \right)^k \tau^k \A^k \left( \I + \tau \A \right)^{-1}\,, \quad \tau > 0\,,
			\end{equation*}
			we find that
			\begin{align}\label{eq:W1_expans_undef_coeff}
				\W_1 \left( t \right) &= \left( \I + \bar{\lambda} t^2 \A \right)^{-1} \left( \I + \lambda t^2 \A \right)^{-1}\nonumber \\
				&= \left( \I - \bar{\lambda} t^2 \A + \bar{\lambda}^2 t^4 \A^2 - \bar{\lambda}^3 t^6 \A^3 \left( \I + \bar{\lambda} t^2 \A \right)^{-1}  \right) \left( \I + \lambda t^2 \A \right)^{-1}\nonumber \\
				&= \left( \I + \lambda t^2 \A \right)^{-1} - \bar{\lambda} t^2 \A \left( \I + \lambda t^2 \A \right)^{-1} + \bar{\lambda}^2 t^4 \A^2 \left( \I + \lambda t^2 \A \right)^{-1}\nonumber \\
				&- \bar{\lambda}^3 t^6 \A^3 \left( \I + \bar{\lambda} t^2 \A \right)^{-1} \left( \I + \lambda t^2 \A \right)^{-1}\nonumber \\
				&= \I - \lambda t^2 \A + \lambda^2 t^4 \A^2 - \lambda^3 t^6 \A^3 \left( \I + \lambda t^2 \A \right)^{-1}\nonumber \\
				&- \bar{\lambda} t^2 \A \left( \I - \lambda t^2 \A + \lambda^2 t^4 \A^2 \left( \I + \lambda t^2 \A \right)^{-1} \right)\nonumber \\
				&+ \bar{\lambda}^2 t^4 \A^2 \left( \I - \lambda t^2 \A \left( \I + \lambda t^2 \A \right)^{-1} \right)\nonumber \\
				&- \bar{\lambda}^3 t^6 \A^3 \left( \I + \bar{\lambda} t^2 \A \right)^{-1} \left( \I + \lambda t^2 \A \right)^{-1}\nonumber \\
				&= \I - \left( \lambda + \bar{\lambda} \right) t^2 \A + \left( \lambda^2 + \lambda \bar{\lambda} + \bar{\lambda}^2 \right) t^4 \A^2 + \widetilde{R}_1 \left( t \right)\,,
			\end{align}
			where
			\begin{equation*}
				\widetilde{R}_1 \left( t \right) = - \left( \lambda \left( \lambda^2 + \lambda \bar{\lambda} + \bar{\lambda}^2 \right) \I + \bar{\lambda}^3 \left( \I + \bar{\lambda} t^2 \A \right)^{-1} \right) \left( \I + \lambda t^2 \A \right)^{-1} t^6 \A^3\,.
			\end{equation*}
			
			Applying the same transformation to $\W_2 \left( t \right)$, one obtains
			\begin{equation}\label{eq:W2_expans_undef_coeff}
				\W_2 \left( t \right) = t \A^{1 / 2} - \frac{1}{6} t^3 \A^{3 / 2} - \left( \lambda_0 + \bar{\lambda}_0 \right) t^3 \A^{3 / 2} + \widetilde{R}_2 \left( t \right)\,,
			\end{equation}
			for which
			\begin{equation*}
				\widetilde{R}_2 \left( t \right) = \left( \frac{1}{36} \I + \left( \left( \lambda_0^2 + \lambda_0 \bar{\lambda}_0 \right) \I + \bar{\lambda}_0^2 \left( \I + \bar{\lambda}_0 t^2 \A \right)^{-1} \right) \left( \I + \lambda_0 t^2 \A \right)^{-1} \right) \left( \I + \frac{1}{6} t^2 \A \right)^{-1} t^5 \A^{5 / 2}\,.
			\end{equation*}
			
			Substituting the values of $\lambda$ and $\lambda_0$, we arrive at
			\begin{align}
				\W_1 \left( t \right) &= \I - \frac{1}{2} t^2 \A + \frac{1}{24} t^4 \A^2 + \widetilde{R}_1 \left( t \right)\,,\label{eq:W1_expans} \\
				\W_2 \left( t \right) &= t \A^{1 / 2} \left( \I - \frac{1}{6} t^2 \A \right) + \widetilde{R}_2 \left( t \right)\,.\label{eq:W2_expans}
			\end{align}
			
			For the remainder terms in \eqref{eq:W1_expans} and \eqref{eq:W2_expans}, the following bounds are established:
			\begin{align}
				\left\Vert \widetilde{R}_1 \left( t \right) u \right\Vert &\leq \frac{1}{4} \left\vert t \right\vert^5 \left\Vert \A^{5 / 2} u \right\Vert\,, \quad u \in D \left( \A^{5 / 2} \right)\,,\label{eq:est_tilde_r1} \\
				\left\Vert \widetilde{R}_2 \left( t \right) u \right\Vert &\leq \frac{19}{144} \left\vert t \right\vert^5 \left\Vert \A^{5 / 2} u \right\Vert\,, \quad u \in D \left( \A^{5 / 2} \right)\,.\label{eq:est_tilde_r2}
			\end{align}
			
			Invoking the formula (see Kato \cite{Kato1980}, {\bfseries Chapter IX})
			\begin{equation}\label{eq:kato_integr_form}
				\sqrt{\A} \int_{r}^{t} e^{\pm \ii s \sqrt{\A}} \d s = \mp \ii \left( e^{\pm \ii t \sqrt{\A}} - e^{\pm \ii r \sqrt{\A}} \right) \quad \text{for all} \quad r,t \in \Real\,,
			\end{equation}
			we obtain the expansion
			\begin{equation}\label{eq:exp_expansion}
				e^{\pm \ii t \sqrt{\A}} u = \sum_{j = 0}^{k} \left( \pm \ii \right)^j \frac{t^j}{j!} \left( \sqrt{\A} \right)^j u + \left( \pm \ii \right)^{k + 1} R_k \left( t \right) \left( \sqrt{\A} \right)^{k + 1} u\,, \quad u \in D \left( \A^{\left( k + 1 \right) / 2} \right)\,,
			\end{equation}
			where
			\begin{equation*}
				R_k \left( t \right) = \int_{0}^{t} \int_{0}^{s_1} \cdots \int_{0}^{s_k} e^{\pm \ii s \sqrt{\A}} \,\d s \d s_k \ldots \d s_1\,.
			\end{equation*}
			
			If we substitute representation \eqref{eq:exp_expansion} into the formulas given in \eqref{eq:euler_formulas}, then we derive the following expansions for the operator functions $\cos \left( t \sqrt{\A} \right)$ and $\sin \left( t \sqrt{\A} \right)$, respectively:
			\begin{align}
				\cos \left( t \sqrt{\A} \right) u &= \sum_{k = 0}^{n} \left( -1 \right)^k \frac{t^{2k}}{\left( 2k \right)!} \A^k u + \left( -1 \right)^{n + 1} \hat{R}_{2n} \left( t \right) \A^{n + 1/2} u\,, \quad u \in D \left( \A^{n + 1/2} \right)\,,\label{eq:cos_oper_expans} \\
				\sin \left( t \sqrt{\A} \right) u &= \sum_{k = 0}^{n} \left( -1 \right)^k \frac{t^{2k + 1}}{\left( 2k + 1 \right)!} \A^{k + 1/2} u + \left( -1 \right)^{n + 1} \hat{R}_{2n + 1} \left( t \right) \A^{n + 1} u\,, \quad u \in D \left( \A^{n + 1} \right)\,,\label{eq:sin_oper_expans}
			\end{align}
			in which
			\begin{equation*}
				\hat{R}_n \left( t \right) = \int_{0}^{t} \int_{0}^{s_1} \cdots \int_{0}^{s_n} \sin \left( s \sqrt{\A} \right) \,\d s \d s_n \ldots \d s_1\,.
			\end{equation*}
			
			The expansions obtained above represent the operator analogues of the power series expansions of the corresponding scalar functions.
			
			Since $e^{\pm \ii t \sqrt{\A}}$ are unitary operators, we have $\left\Vert e^{\pm \ii t \sqrt{\A}} \right\Vert = 1$. From this, it follows that $\left\Vert 	\sin \left( t \sqrt{\A} \right)  \right\Vert \leq 1$. Taking this fact into account, we arrive at the following result:
			\begin{equation*}
				\left\Vert \hat{R}_n \left( t \right) \right\Vert \leq \frac{{\left\vert t \right\vert}^{n + 1}}{\left( n + 1 \right)!}\,, \quad \text{for all} \quad t \in \Real\,.
			\end{equation*}
			Using this estimate and representation \eqref{eq:cos_oper_expans} with $n = 2$, we obtain the inequality
			\begin{equation}\label{eq:estim_cos}
				\left\Vert \left( \cos \left( t \A^{1 / 2} \right) - \left( \I - \frac{1}{2} t^2 \A + \frac{1}{24} t^4 \A^2 \right) \right) u \right\Vert \leq \frac{1}{120} \left\vert t \right\vert^5 \left\Vert \A^{5 / 2} u \right\Vert\,, \quad u \in D \left( \A^{5 / 2} \right)\,.
			\end{equation}
			
			From expansion \eqref{eq:sin_oper_expans}, if we set $n = 1$, the following inequality is obtained:
			\begin{equation}\label{eq:estim_sin}
				\left\Vert \left( \sin \left( t \A^{1 / 2} \right) - t \A^{1 / 2} \left( \I - \frac{1}{6} t^2 \A \right) \right) u \right\Vert \leq \left\Vert \hat{R}_3 \left( t \right) \A^{-1/2} \right\Vert \left\Vert \A^{5 / 2} u \right\Vert\,, \quad u \in D \left( \A^{5 / 2} \right)\,.
			\end{equation}
			
			According to \eqref{eq:kato_integr_form}, the following formula holds:
			\begin{equation*}
				\sin \left( t \A^{1 / 2} \right) \A^{-1 / 2} = \int_{0}^{t} \cos \left( s \A^{1 / 2} \right) \d s\,.
			\end{equation*}
			From this, we have
			\begin{equation*}
				\left\Vert \sin \left( t \A^{1 / 2} \right) \A^{-1 / 2} \right\Vert \leq \left\vert t \right\vert\,.
			\end{equation*}
			Substituting this estimate into \eqref{eq:estim_sin}, we conclude that
			\begin{equation}\label{eq:estim_sin2}
				\left\Vert \left( \sin \left( t \A^{1 / 2} \right) - t \A^{1 / 2} \left( \I - \frac{1}{6} t^2 \A \right) \right) u \right\Vert \leq \frac{1}{120} \left\vert t \right\vert^5 \left\Vert \A^{5 / 2} u \right\Vert\,, \quad u \in D \left( \A^{5 / 2} \right)\,.
			\end{equation}
			
			From the relations
			\begin{align*}
				\cos \left( t \A^{1 / 2} \right) - \W_1 \left( t \right) &= \left[ \cos \left( t \A^{1 / 2} \right) - \left( \I - \frac{1}{2} t^2 \A + \frac{1}{24} t^4 \A^2 \right) \right] - \widetilde{R}_1 \left( t \right)\,, \\
				\sin \left( t \sqrt{\A} \right) - \W_2 \left( t \right) &= \left[ \sin \left( t \A^{1 / 2} \right) - t \A^{1 / 2} \left( \I - \frac{1}{6} t^2 \A \right) \right] - \widetilde{R}_2 \left( t \right)\,,
			\end{align*}
			together with estimates \eqref{eq:est_tilde_r1} and \eqref{eq:estim_cos}, and \eqref{eq:est_tilde_r2} and \eqref{eq:estim_sin2}, it follows that the bounds \eqref{eq:cos_approx} and \eqref{eq:sin_approx} hold, respectively.
			
			We now prove that the norm of the sum of the squares of the rational operator functions $\W_1 \left( t \right)$ and $\W_2 \left( t \right)$ does not exceed one. For this purpose, we introduce the rational functions corresponding to these operators:
			\begin{equation*}
				r_1 \left( x \right) = \frac{1}{1 + \frac{1}{2} x^2 + \frac{5}{24} x^4}\,, \quad r_2 \left( x \right) = \frac{x}{\left( 1 + \frac{1}{6} x^2 \right) \left( 1 + \frac{5}{48} x^4 \right)}\,.
			\end{equation*}
			These functions are obtained from the representations of $\W_1 \left( t \right)$ and $\W_2 \left( t \right)$ by replacing $t \A^{1 / 2}$ with the scalar variable $x$.
			
			One readily verifies that
			\begin{equation}\label{eq:sum_sq_scalar_rat_func}
				r_1^2 \left( x \right) + r_2^2 \left( x \right) \leq 1\,, \quad x \in \left[ 0,+\infty \right)\,.
			\end{equation}
			
			Since the norm of a rational operator function evaluated at a self-adjoint operator does not exceed the uniform norm of the corresponding scalar function (see, \eg{}, {\bfseries Chapter VII} in Reed and Simon \cite{ReedSimon1980}), it follows from \eqref{eq:sum_sq_scalar_rat_func} that
			\begin{equation*}
				\left\Vert \W_1^2 \left( t \right) + \W_2^2 \left( t \right) \right\Vert \leq 1\,, \quad \text{for all} \quad t \in \Real\,.
			\end{equation*}
			This completes the proof.
		\end{proof}
		
		\begin{remark}\label{prop:remark_norm_v}
			It follows from \lemref{prop:lemma1} and \lemref{prop:lemma2} that the operator $V \left( t \right)$ is contractive on $\Hilb \times \Hilb$ for all $t \in \Real$; that is,
			\begin{equation*}
				\left\Vert V \left( t \right) \right\Vert_{\Hilb \times \Hilb} \leq 1\,.
			\end{equation*}
			Consequently, the approximate transition operator is nonexpansive with respect to the norm of the state vector. This property is essential for establishing the stability and convergence of the approximation scheme developed below.
		\end{remark}
		
		\subsection{Approximation of a Unitary Operator Group with a Framework of Rational Matrix Operator Function}
		
		We have already prepared the ground for approximating the unitary operator group $U \left( t \right)$ through the operators $V \left( t \right)$. We state the following theorem.
		\begin{theorem}\label{prop:thrm_unitary_aprox}
			The sequence of operators $\left\{ \left( V \left( \frac{t}{n} \right) \right)^n \right\}_{n = 1}^{\infty}$ approximates the unitary operator group $U \left( t \right)$ in the strong operator topology. That is,
			\begin{equation}\label{eq:converg_over_Hilb}
				\lim_{n \to \infty} \left( V \left( \frac{t}{n} \right) \right)^n \boldsymbol{u} = U \left( t \right) \boldsymbol{u} \quad \text{for all} \quad \boldsymbol{u} = \left\{ u_1,u_2 \right\} \in \Hilb \times \Hilb\,.
			\end{equation}
			Moreover, the following estimate holds
			\begin{equation}\label{eq:diff_U_V_global}
				\left\Vert \left[ \left( V \left( \frac{t}{n} \right) \right)^n - U \left( t \right) \right] \boldsymbol{u} \right\Vert_{\Hilb \times \Hilb} \leq \left( \frac{t}{n} \right)^4 \left( \frac{5}{12} \left\vert t \right\vert \right) \left\Vert \left\{ \A^{5 / 2} u_1,\A^{5 / 2} u_2 \right\} \right\Vert_{\Hilb \times \Hilb}\,,
			\end{equation}
			where $u_1$ and $u_2$ belong to $D \left( \A^{5 / 2} \right)$.
		\end{theorem}
		\begin{proof}
			Let us introduce the vectors $\left( U \left( \tau \right) \right)^k \boldsymbol{u}$ and $\left( V \left( \tau \right) \right)^k \boldsymbol{u}$ with $k = 1,2,\ldots n$ and $\tau = t / n$. Since the family the operators $\left\{ U \left( t \right) \right\}_{t \in \Real}$ forms a group, then we have $\left( U \left( \tau \right) \right)^k = U \left( t_k \right)$ where $t_k = k \tau$. By taking into account this fact, we readily obtain the following representation:
			\begin{align}\label{eq:diff_Uk_Vk}
				\left[ U \left( t_k \right) - \left( V \left( \tau \right) \right)^k \right] \boldsymbol{u} &= \left[ \left( U \left( \tau \right) \right)^k - \left( V \left( \tau \right) \right)^k \right] \boldsymbol{u}\nonumber \\
				&= \sum_{i = 1}^{k} \left( V \left( \tau \right) \right)^{k - i} \left(  U \left( \tau \right) -  V \left( \tau \right) \right) \left( U \left( \tau \right) \right)^{i - 1} \boldsymbol{u}\,.
			\end{align}
			
			If we take into account that $\left\Vert V \left( \tau \right) \right\Vert_{\Hilb \times \Hilb} \leq 1$ (see \remref{prop:remark_norm_v}), $\left\Vert U \left( \tau \right) \right\Vert_{\Hilb \times \Hilb} = 1$, and the operators $U \left( \tau \right)$ and $V \left( \tau \right)$ are commutative, then from \eqref{eq:diff_Uk_Vk} we obtain the following inequality:
			\begin{equation}\label{eq:diff_U_V}
				\left\Vert \left[ U \left( t_k \right) - \left( V \left( \tau \right) \right)^k \right] \boldsymbol{u} \right\Vert_{\Hilb \times \Hilb} \leq k \left\Vert \left(  U \left( \tau \right) -  V \left( \tau \right) \right) \boldsymbol{u} \right\Vert_{\Hilb \times \Hilb}\,.
			\end{equation}
			
			Estimate the right-hand side of inequality \eqref{eq:diff_U_V}. It is evident that we have:
			\begin{align*}
				\left(  U \left( \tau \right) -  V \left( \tau \right) \right) &\boldsymbol{u} \\
				&=
				\left(
				\begin{array}{c}
					\left( \cos \left( \tau \A^{1 / 2} \right) - \W_1 \left( \tau \right) \right) u_1 + \left( \sin \left( \tau \A^{1 / 2} \right) - \W_2 \left( \tau \right) \right) u_2 \\
					\left( \W_2 \left( \tau \right) - \sin \left( \tau \A^{1 / 2} \right) \right) u_1 + \left( \cos \left( \tau \A^{1 / 2} \right) - \W_1 \left( \tau \right) \right) u_2
				\end{array}
				\right)\,.
			\end{align*}
			From here, we have:
			\begin{align*}
				\left\Vert \left(  U \left( \tau \right) -  V \left( \tau \right) \right) u \right\Vert_{\Hilb \times \Hilb}^2 \leq &2 \left( \left\Vert \left( \cos \left( \tau \A^{1 / 2} \right) - \W_1 \left( \tau \right) \right) u_1 \right\Vert^2 + \left\Vert \left( \sin \left( \tau \A^{1 / 2} \right) - \W_2 \left( \tau \right) \right) u_2 \right\Vert^2 \right) \\
				+ &2 \left( \left\Vert \left( \W_2 \left( \tau \right) - \sin \left( \tau \A^{1 / 2} \right) \right) u_1 \right\Vert^2 + \left\Vert \left( \cos \left( \tau \A^{1 / 2} \right) - \W_1 \left( \tau \right) \right) u_2 \right\Vert^2 \right)\,.
			\end{align*}
			By substituting estimates \eqref{eq:cos_approx} and \eqref{eq:sin_approx} in the preceding inequality, we establish:
			\begin{equation}\label{eq:diff_U_V_local}
				\left\Vert \left(  U \left( \tau \right) -  V \left( \tau \right) \right) \boldsymbol{u} \right\Vert_{\Hilb \times \Hilb} \leq \frac{5}{12} \left\vert \tau \right\vert^5 \left\Vert \left\{ \A^{5 / 2} u_1,\A^{5/2} u_2 \right\} \right\Vert_{\Hilb \times \Hilb}\,.
			\end{equation}
			
			From inequality \eqref{eq:diff_U_V}, taking into account \eqref{eq:diff_U_V_local}, we arrive at the following estimate:
			\begin{equation}\label{eq:diff_U_V_global_general}
				\left\Vert \left[ U \left( t_k \right) - \left( V \left( \tau \right) \right)^k \right] \boldsymbol{u} \right\Vert_{\Hilb \times \Hilb} \leq \tau^4 \left( \frac{5}{12} \left\vert t_k \right\vert \right) \left\Vert \left\{ \A^{5 / 2} u_1,\A^{5/2} u_2 \right\} \right\Vert_{\Hilb \times \Hilb}\,, \quad k = 1,2,\ldots,n\,.
			\end{equation}
			Obviously, from \eqref{eq:diff_U_V_global_general} for $k = n$ follows \eqref{eq:diff_U_V_global}.
			
			Let us now show the validity of relation \eqref{eq:converg_over_Hilb}. It is evident that from \eqref{eq:diff_U_V_global} follows \eqref{eq:converg_over_Hilb} if the components of the vector $\boldsymbol{u}$ belong to $D \left( \A^{5 / 2} \right)$. If we demonstrate that $D \left( \A^{5 / 2} \right)$ is dense in $\Hilb$, then relation \eqref{eq:converg_over_Hilb} can be extended to the entire Cartesian product space $\Hilb \times \Hilb$. It is not difficult to prove that if the range of a bounded operator acting on $\Hilb$ is dense in $\Hilb$, then it maps a dense set to a dense set. From this fact follows that the set $D \left( \A^2 \right) = \A^{-1} D \left( \A \right)$ is dense in $\Hilb$, since the range of $\A^{-1}$ coincides to $D \left( \A \right)$. From here, in turn, implies that $D \left( \A^{5 / 2} \right) = \A^{-1 / 2} D \left( \A^2 \right)$ is dense in $\Hilb$.
			
			Let $\boldsymbol{u} = \left\{ u_1,u_2 \right\}$ be an arbitrary vector from $\Hilb \times \Hilb$. Given that $D \left( \A^{5 / 2} \right)$ is dense in $\Hilb$, consequently for any $\varepsilon / 3$ there exist vector $\boldsymbol{v} = \left\{ v_1,v_2 \right\}$ with components belong to $D \left( \A^{5 / 2} \right)$ such that $\left\Vert \boldsymbol{u} - \boldsymbol{v} \right\Vert_{\Hilb \times \Hilb} < \varepsilon / 3$. Since relation \eqref{eq:converg_over_Hilb} holds for any vector from $\Hilb \times \Hilb$ whose components belong to $D \left( \A^{5 / 2} \right)$ therefore there exist a natural number $n_0$ such that, for every natural number $n > n_0$, we have:
			\begin{equation*}
				\left\Vert \left( V \left( \frac{t}{n} \right) \right)^n \boldsymbol{v} - U \left( t \right) \boldsymbol{v} \right\Vert_{\Hilb \times \Hilb} < \frac{\varepsilon}{3}\,.
			\end{equation*}
			Evidently, from the following equality
			\begin{align*}
				\left( V \left( \frac{t}{n} \right) \right)^n \boldsymbol{u} - U \left( t \right) \boldsymbol{u} = &\left[ \left( V \left( \frac{t}{n} \right) \right)^n \boldsymbol{u} - \left( V \left( \frac{t}{n} \right) \right)^n \boldsymbol{v} \right] \\
				+ &\left[ \left( V \left( \frac{t}{n} \right) \right)^n \boldsymbol{v} - U \left( t \right) \boldsymbol{v} \right] + \left[ U \left( t \right) \boldsymbol{v} - U \left( t \right) \boldsymbol{u} \right]\,,
			\end{align*}
			by taking into account the above-established fact together with $\left\Vert V \left( t \right) \right\Vert_{\Hilb \times \Hilb} \leq 1$ and $\left\Vert U \left( t \right) \right\Vert_{\Hilb \times \Hilb} = 1$, we conclude
			\begin{equation*}
				\left\Vert \left( V \left( \frac{t}{n} \right) \right)^n \boldsymbol{u} - U \left( t \right) \boldsymbol{u} \right\Vert_{\Hilb \times \Hilb} < \varepsilon\,, \quad \text{when} \quad n > n_0\,.
			\end{equation*}
			Thus, we have shown that the relation \eqref{eq:converg_over_Hilb} is valid for every $\boldsymbol{u} = \left\{ u_1,u_2 \right\} \in \Hilb \times \Hilb$.
		\end{proof}
		
		In view of the estimates established in \lemref{prop:lemma2}, it is natural to investigate the approximation properties under weaker regularity assumptions. More precisely, let $u \in D \left( \A^{2 - i/2} \right)$, with $i = 0,1,2$. One may then ask: what is the order of approximation of the cosine and sine operator functions by the corresponding rational operator functions $\W_1 \left( t \right)$ and $\W_2 \left( t \right)$? The following lemma addresses this question.
		\begin{lemma}\label{prop:lemma3}
			For every $u \in D \left( \A^{2 - i/2} \right)$, with $i = 0,1,2$, the following estimates hold:
			\begin{align*}
				\left\Vert \left( \cos \left( t \A^{1 / 2} \right) - \W_1 \left( t \right) \right) u \right\Vert &\leq c_{1 + 2i} \left\vert t \right\vert^{4 - i} \left\Vert \A^{2 - i/2} u \right\Vert\,, \\
				\left\Vert \left( \sin \left( t \A^{1 / 2} \right) - \W_2 \left( t \right) \right) u \right\Vert &\leq c_{2 + 2i} \left\vert t \right\vert^{4 - i} \left\Vert \A^{2 - i/2} u \right\Vert\,.
			\end{align*}
			Here, the constants are given by $c_1 = 2/3$, $c_2 = 11/24$, $c_3 = 7/6$, $c_4 = 1$, $c_5 = 3/2$, and $c_6 = 11/6$.
		\end{lemma}
		\begin{proof}
			The proof of \lemref{prop:lemma3} runs analogously to that of \lemref{prop:lemma2}.
		\end{proof}
		
		Arguing as in \thmref{prop:thrm_unitary_aprox} and employing \lemref{prop:lemma3}, one obtains the following result.
		\begin{theorem}\label{prop:thm_addit_order}
			The sequence of operators $\left\{ \left( V \left( \frac{t}{n} \right) \right)^n \right\}_{n = 1}^{\infty}$ provides an approximation of the unitary group $U \left( t \right)$ on the domain $D \left( \A^{2 - i/2} \right) \times D \left( \A^{2 - i/2} \right)$, for $i = 0,1,2$, with convergence of order $\bigO \left( n^{i - 3} \right)$. More precisely, the following estimates hold:
			\begin{equation*}
				\left\Vert \left[ \left( V \left( \frac{t}{n} \right) \right)^n - U \left( t \right) \right] \boldsymbol{u} \right\Vert_{\Hilb \times \Hilb} \leq \left( \frac{\left\vert t \right\vert}{n} \right)^{3 - i} \left( c_{7 + i} \left\vert t \right\vert \right) \left\Vert \left\{ \A^{2 - i/2} u_1,\A^{2 - i/2} u_2 \right\} \right\Vert_{\Hilb \times \Hilb}\,,
			\end{equation*}
			where $c_7 = 5 / 4$, $c_8 = 7 / 3$, and $c_9 = 11 / 3$.
		\end{theorem}
		
		Theorems \ref{prop:thrm_unitary_aprox} and \ref{prop:thm_addit_order} characterize the relationship between the scale of smoothness and the corresponding order of convergence for smoothness levels not exceeding $5/2$. It is therefore natural to consider the behavior in the regime of higher regularity. In this context, the following observation holds and is formulated as a remark.
		\begin{remark}\label{prop:remark_regularity}
			The order of convergence in estimate \eqref{eq:diff_U_V_global} is optimal and cannot be improved. In particular, even under stronger smoothness assumptions beyond $D \left( \A^{5 / 2} \right)$, the convergence order in \eqref{eq:diff_U_V_global} does not exceed four. This limitation is a consequence of the local approximation properties of the associated scalar rational functions $r_1 \left( x \right)$ and $r_2 \left( x \right)$, corresponding to the operator-valued functions $\W_1 \left( t \right)$ and $\W_2 \left( t \right)$. Specifically, in a neighborhood of $x = 0$, the functions $r_1 \left( x \right)$ and $r_2 \left( x \right)$ approximate $\cos \left( x \right)$ and $\sin \left( x \right)$, respectively, with maximal order $\bigO \left( x^5 \right)$.
		\end{remark}
		
		\subsection{Approximation in the Case of Minimal Regularity}
		
		In our opinion, from the point of view of practical application, the issue of approximating the unitary operator group $U \left( t \right)$ by means of a rational matrix operator function under minimal regularity assumptions is important. In particular, on the Cartesian product $D \left( \A^{1/2} \right) \times D \left( \A^{1/2} \right) \subset \Hilb \times \Hilb$.
		
		\begin{theorem}
			The sequence of operators $\left\{ \left( V \left( \frac{t}{n} \right) \right)^n \right\}_{n = 1}^{\infty}$ approximates the unitary operator group $U \left( t \right)$ on the domain $D \left( \A^{1/2} \right) \times D \left( \A^{1/2} \right)$, with convergence order $\bigO\left( n^{-1/8} \right)$. Specifically, the following estimate holds:
			\begin{equation}\label{eq:U_minus_V_global}
				\left\Vert \left[ \left( V \left( \frac{t}{n} \right) \right)^n - U \left( t \right) \right] \boldsymbol{u} \right\Vert_{\Hilb \times \Hilb} \leq \left( \frac{\left\vert t \right\vert}{n} \right)^{1/8} a \left( t \right) \left\Vert \left\{ \A^{1 / 2} u_1,\A^{1 / 2} u_2 \right\} \right\Vert_{\Hilb \times \Hilb}\,,
			\end{equation}
			here, $u_1$ and $u_2$ belong to $D \left( \A^{1 / 2} \right)$, and $a \left( t \right) = 3 \left( 4 \sqrt{2} + 1 \right) \left\vert t \right\vert^{7/8}$.
		\end{theorem}
		\begin{proof}
			Let us introduce the following matrix operator:
			\begin{equation*}
				A =
				\left(
				\begin{array}{rr}
					\A & 0 \\
					0 & \A
				\end{array}
				\right)\,.
			\end{equation*}
			Using this denotation and the property of the unitary operator $\left\Vert \left( U \left( \tau \right) \right)^k \right\Vert = \left\Vert \left( U \left( t_k \right) \right) \right\Vert = 1$, from representation \eqref{eq:diff_Uk_Vk} it follows that
			\begin{align}\label{eq:approx_domain_sqrt_A}
				&\left\Vert \left[ U \left( t_k \right) - \left( V \left( \tau \right) \right)^k \right] \boldsymbol{u} \right\Vert_{\Hilb \times \Hilb}\nonumber \\
				&\leq \sum_{i = 1}^{k - 1} \left\Vert \left(  U \left( \tau \right) -  V \left( \tau \right) \right) A^{-1} \right\Vert_{\Hilb \times \Hilb} \left\Vert A^{1/2} \left( V \left( \tau \right) \right)^{k - i} \right\Vert_{\Hilb \times \Hilb} \left\Vert A^{1/2} \boldsymbol{u} \right\Vert_{\Hilb \times \Hilb}\nonumber \\
				&+ \left\Vert \left(  U \left( \tau \right) -  V \left( \tau \right) \right) \boldsymbol{u}  \right\Vert_{\Hilb \times \Hilb}\,, \quad \text{with} \quad \tau = \frac{t}{n}\,.
			\end{align}
			
			By performing standard algebraic manipulations, analogous to those used in the proof of inequality \eqref{eq:diff_U_V_local}, one obtains:
			\begin{equation*}
				\left\Vert \left(  U \left( \tau \right) -  V \left( \tau \right) \right) \boldsymbol{u} \right\Vert_{\Hilb \times \Hilb} \leq \frac{11}{3} \tau^2 \left\Vert A \boldsymbol{u} \right\Vert_{\Hilb \times \Hilb}\,.
			\end{equation*}
			Hence, we have:
			\begin{equation}\label{eq:diff_UV_inv_A}
				\left\Vert \left(  U \left( \tau \right) -  V \left( \tau \right) \right) A^{-1} \right\Vert_{\Hilb \times \Hilb} \leq \frac{11}{3} \tau^2\,.
			\end{equation}
			
			Let us estimate the norm of the operator $A^{1/2} \left( V \left( \tau \right) \right)^k$. For this, we need one simple fact. Let $C$ be the following second-order matrix with entries of real numbers $x$ and $y$:
			\begin{equation*}
				C =
				\left(
				\begin{array}{rr}
					x & y \\
					-y & x
				\end{array}
				\right)\,.
			\end{equation*}
			Using de Moivre's formula, it is easy to obtain:
			\begin{equation*}
				C^n =
				\left(
				\begin{array}{rr}
					\Re \left( z^n \right) & \Im \left( z^n \right) \\
					-\Im \left( z^n \right) & \Re \left( z^n \right)
				\end{array}
				\right)\,, \quad \text{with} \quad z = x + \ii y\,.
			\end{equation*}
			Considering this fact, we have:
			\begin{equation*}
				 \left( V \left( \tau \right) \right)^k =
				 \left(
				 \begin{array}{rr}
				 	\Re \left( \left( \W \left( \tau \right) \right)^k \right) & \Im \left( \left( \W \left( \tau \right) \right)^k \right) \\
				 	-\Im \left( \left( \W \left( \tau \right) \right)^k \right) & \Re \left( \left( \W \left( \tau \right) \right)^k \right)
				 \end{array}
				 \right)\,, \quad \text{with} \quad \W \left( \tau \right) = \W_1 \left( \tau \right) + \ii \W_2 \left( \tau \right)\,.
			\end{equation*}
			By virtue of this representation, we find that
			\begin{equation}\label{eq:sqrt_Av}
				\left\Vert A^{1/2} \left( V \left( \tau \right) \right)^k \boldsymbol{u} \right\Vert_{\Hilb \times \Hilb} \leq \sqrt{2} \left( \left\Vert \A^{1/2} \Re \left( \left( \W \left( \tau \right) \right)^k \right) \right\Vert + \left\Vert \A^{1/2} \Im \left( \left( \W \left( \tau \right) \right)^k \right) \right\Vert \right) \left\Vert \boldsymbol{u} \right\Vert_{\Hilb \times \Hilb}\,.
			\end{equation}
			
			As it is known, the norm of an operator function, when the argument of the function represents a self-adjoint bounded operator, is less than or equal to the uniform norm of the corresponding scalar function (see, \eg{}, \cite{ReedSimon1980}, {\bfseries Chapter VII}). Using this fact, we arrive at the inequality:
			\begin{align}\label{eq:re_w}
				\left\Vert \A^{1/2} \Re \left( \left( \W \left( \tau \right) \right)^k \right) \right\Vert &\leq \frac{1}{\left\vert \tau \right\vert} \sup_{x \geq 0} \left[ x \left\vert \Re \left( \left( r_1 \left( x \right) + \ii r_2 \left( x \right) \right)^k \right) \right\vert \right]\nonumber \\
				&\leq \frac{1}{\left\vert \tau \right\vert} \sup_{x \geq 0} \left( x \left( r_1^2 \left( x \right) + r_2^2 \left( x \right) \right)^{k/2} \right)\,.
			\end{align}
			
			Performing routine calculations, we obtain the following inequality:
			\begin{equation}\label{eq:sum_sq_r_ineq}
				r_1^2 \left( x \right) + r_2^2 \left( x \right) \leq \frac{1}{1 + a_0 x^8}\,, \quad \text{where} \quad a_0 = \frac{1}{125}\,.
			\end{equation}
			
			Note that, for the rational function $1 / \left( 1 + a x^n \right)$ with $a > 0$, the smallest value of the exponent $n$ for which inequality \eqref{eq:sum_sq_r_ineq} holds is $n = 8$, whereas the largest admissible value of the coefficient $a$ is $a_0$. Furthermore, $a_0$ is optimal to within $10^{-4}$.
			
			If we substitute inequality \eqref{eq:sum_sq_r_ineq} into \eqref{eq:re_w}, it follows that:
			\begin{equation}\label{eq:re_w_2}
				\left\Vert \A^{1/2} \Re \left( \left( \W \left( \tau \right) \right)^k \right) \right\Vert \leq \frac{1}{\left\vert \tau \right\vert} \sup_{x \geq 0} \frac{x}{\left( 1 + a_0 x^8 \right)^{k/2}}\,, \quad k \geq 1\,.
			\end{equation}
			If we apply Bernoulli's inequality for $k \geq 2$ to the function on the right-hand side of inequality \eqref{eq:re_w_2}, we obtain:
			\begin{equation}\label{eq:Bernoulli}
				\sup_{x \geq 0} \frac{x}{\left( 1 + a_0 x^8 \right)^{k/2}} \leq \sup_{x \geq 0} \frac{x}{1 + \frac{k}{2} a_0 x^8} \leq \frac{7}{8} \sqrt[8]{\frac{250}{7}} \frac{1}{\sqrt[8]{k}}\,. 
			\end{equation}
			For $k = 1$, the function on the right-hand side of inequality \eqref{eq:re_w_2} does not exceed $2 / \sqrt[8]{15}$ when $x \geq 0$. Taking this into account, together with estimate \eqref{eq:Bernoulli}, we finally have from \eqref{eq:re_w_2}:
			\begin{equation}\label{eq:re_w_final}
				\left\Vert \A^{1/2} \Re \left( \left( \W \left( \tau \right) \right)^k \right) \right\Vert \leq \frac{a_1}{\left\vert \tau \right\vert} \frac{1}{\sqrt[8]{k}}\,, \quad a_1 = \frac{2}{\sqrt[8]{15}}\,.
			\end{equation}
			
			It should be noted that inequality \eqref{eq:re_w_final} also holds for the operator $A^{1/2} \Im \left( \left( \W \left( \tau \right) \right)^k \right)$. Taking this into account, together with bound \eqref{eq:re_w_final}, the following estimate follows from \eqref{eq:sqrt_Av}:
			\begin{equation}\label{eq:final_sqrt_Av}
				\left\Vert A^{1/2} \left( V \left( \tau \right) \right)^k \right\Vert_{\Hilb \times \Hilb} \leq \frac{2 \sqrt{2} a_1}{\left\vert \tau \right\vert} \frac{1}{\sqrt[8]{k}}\,.
			\end{equation}
			
			Taking into account estimates \eqref{eq:diff_UV_inv_A} and \eqref{eq:final_sqrt_Av} in \eqref{eq:approx_domain_sqrt_A} yields:
			\begin{align}\label{eq:prior_final_est}
				\left\Vert \left[ U \left( t_k \right) - \left( V \left( \tau \right) \right)^k \right] \boldsymbol{u} \right\Vert_{\Hilb \times \Hilb} &\leq \left( \frac{176 \sqrt{2}}{21} a_1 \left\vert t_k \right\vert^{7/8} \right) \left\vert \tau \right\vert^{1/8} \left\Vert A^{1/2} \boldsymbol{u} \right\Vert_{\Hilb \times \Hilb}\nonumber \\
				&+ \left\Vert \left(  U \left( \tau \right) -  V \left( \tau \right) \right) \boldsymbol{u}  \right\Vert_{\Hilb \times \Hilb}\,.
			\end{align}
			
			It remains to estimate the second term appearing on the right-hand side of inequality \eqref{eq:prior_final_est}. Let us represent the difference $\cos \left( \tau \A^{1 / 2} \right) - \W_1 \left( \tau \right)$ as follows:
			\begin{equation}\label{eq:cosine_minus_W}
				\cos \left( \tau \A^{1 / 2} \right) - \W_1 \left( \tau \right) = \left( \cos \left( \tau \A^{1 / 2} \right) - \I \right) + \left( \I - \W_1 \left( \tau \right) \right)\,.
			\end{equation}
			
			Obviously, the following representation holds:
			\begin{equation*}
				\cos \left( \tau \A^{1 / 2} \right) - \I = - \A^{1/2} \int_{0}^{\tau} \sin \left( s \A^{1 / 2} \right) \d s\,.
			\end{equation*}
			Hence, we have
			\begin{equation}\label{eq:cosine_minus_id}
				\left\Vert \left( \cos \left( \tau \A^{1 / 2} \right) - \I \right) \varphi \right\Vert \leq \left\vert \tau \right\vert \left\Vert \A^{1 / 2} \varphi \right\Vert\,, \quad \varphi \in D \left( \A^{1 / 2} \right)\,.
			\end{equation}
			
			Let us estimate the vector $\left( \I - \W_1 \left( \tau \right) \right) \varphi$. By performing standard algebraic manipulations and taking into account that the norm of an operator-valued function, whose argument is a bounded self-adjoint operator, does not exceed the uniform norm of the corresponding scalar function, one obtains:
			\begin{align}\label{eq:id_minus_W}
				\left\Vert \left( \I - \W_1 \left( \tau \right) \right) \varphi \right\Vert &\leq \left\vert \tau \right\vert \left\Vert \left\vert \tau \right\vert \A^{1/2} \left( \I + \lambda \tau^2 \A \right)^{-1} \left[ \lambda \I + \bar{\lambda} \left( \I + \bar{\lambda} \tau^2 \A \right)^{-1} \right] \right\Vert \left\Vert \A^{1/2} \varphi \right\Vert\nonumber \\
				&\leq \left\vert \tau \right\vert \sup_{x \geq 0} \frac{\frac{1}{2} x + \frac{5}{24} x^3}{1 + \frac{1}{2} x^2 + \frac{5}{24} x^4} \left\Vert \A^{1/2} \varphi \right\Vert \leq \frac{\left\vert \tau \right\vert}{2} \left\Vert \A^{1/2} \varphi \right\Vert\,.
			\end{align}
			
			From \eqref{eq:cosine_minus_W}, taking into account the estimates \eqref{eq:cosine_minus_id} and \eqref{eq:id_minus_W}, we obtain:
			\begin{equation}\label{eq:est_cosine_minus_W}
				\left\Vert \left( \cos \left( \tau \A^{1 / 2} \right) - \W_1 \left( \tau \right) \right) \varphi \right\Vert \leq \frac{3}{2} \left\vert \tau \right\vert \left\Vert \A^{1/2} \varphi \right\Vert\,, \quad \varphi \in D \left( \A^{1 / 2} \right)\,.
			\end{equation}
			
			We now estimate the difference $\sin \left( \tau \A^{1 / 2} \right) - \W_2 \left( \tau \right)$. To this end, we represent it in the following form:
			\begin{equation}\label{eq:sine_minus_W}
				\sin \left( \tau \A^{1 / 2} \right) - \W_2 \left( \tau \right) = \left( \sin \left( \tau \A^{1 / 2} \right) - \tau \A^{1 / 2} \right) + \left( \tau \A^{1 / 2} - \W_2 \left( \tau \right) \right)\,.
			\end{equation}
			
			A straightforward computation yields:
			\begin{equation}\label{eq:sine_minus_arg}
				\left\Vert \left( \sin \left( \tau \A^{1 / 2} \right) - \tau \A^{1 / 2} \right) \varphi \right\Vert = \frac{1}{2} \left\Vert \int_{0}^{\tau} \sin^2 \left( \frac{s}{2} \A^{1 / 2} \right) \A^{1 / 2}\varphi\, \d s \right\Vert \leq \frac{\left\vert \tau \right\vert}{2} \left\Vert \A^{1/2} \varphi \right\Vert\,.
			\end{equation}
			
			Proceeding analogously to the derivation of \eqref{eq:id_minus_W}, we obtain:
			\begin{equation}\label{eq:arg_minus_W}
				\left\Vert \left( \tau \A^{1 / 2} - \W_2 \left( \tau \right) \right) \varphi \right\Vert \leq \left\vert \tau \right\vert \left\Vert \A^{1 / 2} \varphi \right\Vert\,.
			\end{equation}
			
			Combining relation \eqref{eq:sine_minus_W} with estimates \eqref{eq:sine_minus_arg} and \eqref{eq:arg_minus_W}, we arrive at:
			\begin{equation}\label{eq:est_sine_minus_W}
				\left\Vert \left( \sin \left( \tau \A^{1 / 2} \right) - \W_2 \left( \tau \right) \right) \varphi \right\Vert \leq \frac{3}{2} \left\vert \tau \right\vert \left\Vert \A^{1/2} \varphi \right\Vert\,, \quad \varphi \in D \left( \A^{1 / 2} \right)\,.
			\end{equation}
			
			We are now able to estimate the second term appearing on the right-hand side of inequality \eqref{eq:prior_final_est}. Using inequalities \eqref{eq:est_cosine_minus_W} and \eqref{eq:est_sine_minus_W}, we obtain:
			\begin{equation}\label{eq:second_term_final}
				\left\Vert \left(  U \left( \tau \right) -  V \left( \tau \right) \right) \boldsymbol{u}  \right\Vert_{\Hilb \times \Hilb} \leq 3 \left\vert \tau \right\vert \left\Vert A^{1/2} \boldsymbol{u} \right\Vert_{\Hilb \times \Hilb}\,.
			\end{equation}
			
			Upon substituting inequality \eqref{eq:second_term_final} into \eqref{eq:prior_final_est}, we arrive at the final estimate:
			\begin{equation}\label{eq:U_minus_V_general}
				\left\Vert \left[ U \left( t_k \right) - \left( V \left( \tau \right) \right)^k \right] \boldsymbol{u} \right\Vert_{\Hilb \times \Hilb} \leq a \left( t_k \right) \left\vert \tau \right\vert^{1/8} \left\Vert A^{1/2} \boldsymbol{u} \right\Vert_{\Hilb \times \Hilb}\,.
			\end{equation}
			Relation \eqref{eq:U_minus_V_global} follows from \eqref{eq:U_minus_V_general} by taking $k = n$.
		\end{proof}
		
		\section{Approximate Solution to the Cauchy Problem for an Abstract Hyperbolic Equation}\label{sec:approx_sol_abstr_hyper_eq}
		
		\subsection{Linear Problem}\label{subsec:lin_probl}
		
		Consider a Cauchy problem for an abstract hyperbolic equation in Hilbert space $\Hilb$:
		\begin{gather}
			\frac{\d^2 u \left( t \right)}{\d t^2} + \A u \left( t \right) = f \left( t \right)\,, \quad t \in \left[ 0,T \right]\,,\label{eq:abstr_lin_hiperb} \\
			u \left( 0 \right) = \varphi_0\,, \quad \left.\frac{\d u \left( t \right)}{\d t} \right|_{t = 0} = \varphi_1\,.\label{eq:init_data_lin_hiperb}
		\end{gather}
		Here, $\A$ is a densely defined, self-adjoint, positive-definite operator (not necessarily bounded) with domain $D \left( \A \right)$. That is, $\overline{D \left( \A \right)} = \Hilb$, $\A = \A^{\ast}$, and
		\begin{equation*}
			\left( \A u,u \right) \geq \alpha \left\Vert u \right\Vert^2\,, \quad \forall u \in D \left( \A \right)\,, \quad \alpha = \mathrm{const} > 0\,.
		\end{equation*}
		$\varphi_0$ and $\varphi_1$ are given vectors in $\Hilb$; $f \left( t \right)$ is a continuous function with values from $\Hilb$; $u \left( t \right)$ is a continuous and twice continuously differentiable unknown function that takes values from $\Hilb$. In this context, continuity and differentiability are understood in the sense of a metric defined on the Hilbert space.
		
		It is known that if $\varphi_0 \in D \left( \A \right)$, $\varphi_1 \in D \left( \A^{1 / 2} \right)$, and $f \left( t \right)$ is continuous and continuously differentiable function, then there exist such twice continuously differentiable function $u \left( t \right)$ which satisfies equation \eqref{eq:abstr_lin_hiperb} with initial conditions \eqref{eq:init_data_lin_hiperb} (see Kre\u{\i}n \cite{Krein1971}, {\bfseries Chapter III}). In this case, the solution to problem \eqref{eq:abstr_lin_hiperb} is represented in the following way:
		\begin{equation}\label{eq:sol_abstr_hyper}
			u \left( t \right) = \cos\left( t \A^{1 / 2} \right) \varphi_0 + \sin\left( t \A^{1 / 2} \right) \A^{-1 / 2} \varphi_1 + \int_{0}^{t} \sin\left( \left( t - s \right) \A^{1 / 2} \right) \A^{-1 / 2} f \left( s \right) \d s\,.
		\end{equation}
		
		In view of \eqref{eq:sol_abstr_hyper}, we have:
		\begin{equation}\label{eq:vect_repres_sol_abstr_hyper}
			\boldsymbol{w} \left( t \right) = U \left( t \right) \boldsymbol{w} \left( 0 \right) + \int_{0}^{t} U \left( t - s \right) \begin{pmatrix} 0 \\ f \left( s \right) \end{pmatrix} \d s\,.
		\end{equation}
		where $\boldsymbol{w} \left( t \right) = \left( \A^{1 / 2} u \left( t \right),u^{\prime} \left( t \right) \right)^{\top}$.
		
		An important consequence of representation \eqref{eq:vect_repres_sol_abstr_hyper} is that it enables us to establish the conservation of energy. This fact is made precise in the following lemma.
		\begin{lemma}\label{prop:lemma_energy}
			In the homogeneous case of problem \eqref{eq:abstr_lin_hiperb}–\eqref{eq:init_data_lin_hiperb}, that is, when $f \left( t \right) \equiv 0$, the energy vector $\boldsymbol{w} \left( t \right) = \left( \A^{1 / 2} u \left( t \right),u^{\prime} \left( t \right) \right)^{\top}$ satisfies the following identity:
			\begin{equation*}
				\left\Vert \A^{1 / 2} u \left( t \right) \right\Vert^2 + \left\Vert u^{\prime} \left( t \right) \right\Vert^2 = \left\Vert \A^{1 / 2} \varphi_0 \right\Vert^2 + \left\Vert \varphi_1 \right\Vert^2\,.
			\end{equation*}
			Thus, energy is conserved.
		\end{lemma}
		\begin{proof}
			By \eqref{eq:vect_repres_sol_abstr_hyper}, $f \left( t \right) \equiv 0$ implies $\boldsymbol{w} \left( t \right) = U \left( t \right) \boldsymbol{w} \left( 0 \right)$. Since $U \left( t \right)$ is a unitary operator group, it follows immediately that:
			\begin{align*}
				\left\Vert \boldsymbol{w} \left( t \right) \right\Vert_{\Hilb \times \Hilb}^2 &= \left( \left( U \left( t \right) \boldsymbol{w} \left( 0 \right),U \left( t \right) \boldsymbol{w} \left( 0 \right) \right) \right) \\
				&= \left( \left( U^{\ast} \left( t \right) U \left( t \right) \boldsymbol{w} \left( 0 \right), \boldsymbol{w} \left( 0 \right) \right) \right) \\
				&= \left\Vert \boldsymbol{w} \left( 0 \right) \right\Vert_{\Hilb \times \Hilb}^2 = \left\Vert \A^{1 / 2} \varphi_0 \right\Vert^2 + \left\Vert \varphi_1 \right\Vert^2\,.
			\end{align*}
			This proves the desired equality. Recall that $\left( \left( \cdot,\cdot \right) \right)$ denotes the inner product on the product space $\Hilb \times \Hilb$.
		\end{proof}
		Observe that \lemref{prop:lemma_energy} shows that the conservation of energy is inherent in the unitary operator group representation of the energy vector.
		
		Introduce the uniform grid for the interval $\left[ 0,T \right]$ by dividing into $n \geq 2$ equal subintervals. Temporal nodes are denoted by $t_k$, where $t_k = k \tau$ with $\tau = T / n$ and $k = 0,1,\ldots,n$. We consider problem \eqref{eq:abstr_lin_hiperb}–\eqref{eq:init_data_lin_hiperb} on the subinterval $\left[ t_{k - 1},t_k \right]$. According to the formula \eqref{eq:sol_abstr_hyper}, we have
		\begin{align}
			u \left( t_k \right) = \cos\left( \tau \A^{1 / 2} \right) u \left( t_{k - 1} \right) &+ \sin\left( \tau \A^{1 / 2} \right) \A^{-1 / 2} u^{\prime} \left( t_{k - 1} \right)\nonumber \\
			&+ \int_{t_{k - 1}}^{t_k} \sin\left( \left( t_k - s \right) \A^{1 / 2} \right) \A^{-1 / 2} f \left( s \right) \d s\,,\label{eq:exact_u_tk} \\
			u^{\prime} \left( t_k \right) = - \sin\left( \tau \A^{1 / 2} \right) \A^{1 / 2} u \left( t_{k - 1} \right) &+ \cos\left( \tau \A^{1 / 2} \right) u^{\prime} \left( t_{k - 1} \right)\nonumber \\
			&+ \int_{t_{k - 1}}^{t_k} \cos\left( \left( t_k - s \right) \A^{1 / 2} \right) f \left( s \right) \d s\,.\label{eq:exact_diff_u_tk}
		\end{align}
		
		If the integrals appearing in representations \eqref{eq:exact_u_tk} and \eqref{eq:exact_diff_u_tk} are computed by Simpson's 1/3 rule, we obtain the following relations:
		\begin{align}
			u \left( t_k \right) &= \cos\left( \tau \A^{1 / 2} \right) u \left( t_{k - 1} \right) + \sin\left( \tau \A^{1 / 2} \right) \A^{-1 / 2} u^{\prime} \left( t_{k - 1} \right)\nonumber \\
			&+ \frac{\tau}{6} \A^{-1 / 2} g_{1,k} \left( \tau \right) + \A^{-1 / 2} \bar{R}_{1,k} \left( \tau \right)\,,\label{eq:Simps_u_tk} \\
			u^{\prime} \left( t_k \right) &= - \sin\left( \tau \A^{1 / 2} \right) \A^{1 / 2} u \left( t_{k - 1} \right) + \cos\left( \tau \A^{1 / 2} \right) u^{\prime} \left( t_{k - 1} \right)\nonumber \\
			&+ \frac{\tau}{6} g_{2,k} \left( \tau \right) + \bar{R}_{2,k} \left( \tau \right)\,,\label{eq:Simps_diff_u_tk}
		\end{align}
		where
		\begin{align*}
			g_{1,k} \left( \tau \right) = \sin\left( \tau \A^{1 / 2} \right) f_{k - 1} &+ 4 \sin\left( \frac{\tau}{2} \A^{1 / 2} \right) f_{k - 1/2}\,, \\
			g_{2,k} \left( \tau \right) = \cos\left( \tau \A^{1 / 2} \right) f_{k - 1} &+ 4 \cos\left( \frac{\tau}{2} \A^{1 / 2} \right) f_{k - 1/2} + f_k\,, \quad f_k = f\left( t_k \right)\,.
		\end{align*}
		
		The following estimates hold for the remainder terms included in \eqref{eq:Simps_u_tk} and \eqref{eq:Simps_diff_u_tk}:
		\begin{equation}\label{eq:Simps_error_unified}
			\left\Vert \bar{R}_{j,k} \left( \tau \right) \right\Vert \leq \frac{49}{2880} \tau^5 \sum_{i = 0}^{4} \binom{4}{i} \sup_{t \in \left[t_{k - 1},t_k\right]} \left\Vert \A^{2 - i / 2} f^{\left( i \right)} \left( t \right) \right\Vert \quad \text{for} \quad j = 1,2\,.
		\end{equation}
		Here and throughout, the symbol $\binom{m}{i}$ denotes the binomial coefficient, while the superscript notation $\left( i \right)$ represents the $i$-th order of time derivatives.
		
		It should be noted that the estimate \eqref{eq:Simps_error_unified} is not an extension of the well-known estimate of the remainder term of Simpson's rule for a scalar function to the case of an abstract function. As is known, the estimate of the remainder term of Simpson's rule for a scalar function is based on the mean value theorem, which is not valid for an abstract function in the classical sense. We obtained the estimate \eqref{eq:Simps_error_unified} by performing standard algebraic manipulations using formulas \eqref{eq:cos_oper_expans} and \eqref{eq:sin_oper_expans}, along with the Taylor series expansion for the function $f \left( t \right)$.
		
		If we replace the cosine and sine operator functions in \eqref{eq:Simps_u_tk} and \eqref{eq:Simps_diff_u_tk} with the rational operator functions $\W_1 \left( \tau \right)$ and $\W_2 \left( \tau \right)$, respectively, and neglect the remainder terms $\bar{R}_1 \left( \tau \right)$ and $\bar{R}_2 \left( \tau \right)$, then we obtain the following scheme:
		\begin{align}
			u_k^{\left( 0 \right)} &= \W_1 \left( \tau \right) u_{k - 1}^{\left( 0 \right)} + \W_2 \left( \tau \right) \A^{-1 / 2} u_{k - 1}^{\left( 1 \right)} + \frac{\tau}{6} \A^{-1 / 2} \tilde{g}_{1,k}\,,\label{eq:Simps_approx_u_tk} \\
			u_k^{\left( 1 \right)} &= - \W_2 \left( \tau \right) \A^{1 / 2} u_{k - 1}^{\left( 0 \right)} + \W_1 \left( \tau \right) u_{k - 1}^{\left( 1 \right)} + \frac{\tau}{6} \tilde{g}_{2,k}\,,\label{eq:Simps_diff_approx_u_tk}
		\end{align}
		where $k = 1,2,\ldots,n$, $u_0^{\left( 0 \right)} = \varphi_0$, $u_0^{\left( 1 \right)} = \varphi_1$, and
		\begin{equation*}
			\tilde{g}_{1,k} \left( \tau \right) = \W_2 \left( \tau \right) f_{k - 1} + 4 \W_2 \left( \frac{\tau}{2} \right) f_{k - 1/2}\,, \quad \tilde{g}_{2,k} \left( \tau \right) = \W_1 \left( \tau \right) f_{k - 1} + 4 \W_1 \left( \frac{\tau}{2} \right) f_{k - 1/2} + f_k\,.
		\end{equation*}
		
		It is natural to declare $u_k^{\left( 0 \right)}$ and $u_k^{\left( 1 \right)}$ as the approximate values of the functions $u \left( t \right)$ and $u^{\prime} \left( t \right)$, respectively, at the temporal node points $t = t_k$, \ie{}, $\left\{ u \left( t_k \right),u^{\prime} \left( t_k \right) \right\} \approx \left\{ u_k^{\left( 0 \right)},u_k^{\left( 1 \right)} \right\}$.
		
		Before we state \thmref{prop:thm_convrg_w_v}, let us introduce the following notation:
		\begin{equation*}
			\boldsymbol{w_k} = \left( \A^{1 / 2} u \left( t_k \right),u^{\prime} \left( t_k \right) \right)^{\top} \quad \text{and} \quad \boldsymbol{v_k} = \left( \A^{1 / 2}u_k^{\left( 0 \right)},u_k^{\left( 1 \right)} \right)^{\top}\,.
		\end{equation*}
		
		\begin{theorem}\label{prop:thm_convrg_w_v}
			Let the following conditions be fulfilled: $\varphi_0 \in D \left( \A^3 \right)$ and $\varphi_1 \in D \left( \A^{5/2} \right)$. The right-hand side $f \left( t \right)$ is a continuous and continuously differentiable function up to and including the fourth order. Furthermore, $f \left( t \right) \in D \left( \A^{5/2} \right)$ and $f^{\left( i \right)} \left( t \right) \in D \left( \A^{2 - i/2} \right)$ with $i = 1,2,3$ for each fixed value of $t$ from the interval $\left[ 0,T \right]$. Then, for the approximation error, the following estimate holds:
			\begin{align}\label{eq:approx_error_diff_w_v}
				\left\Vert \boldsymbol{w_k} - \boldsymbol{v_k} \right\Vert_{\Hilb \times \Hilb} &\leq \tau^4 \left( \frac{5}{12} t_k \right) \left( \left\Vert \A^3 \varphi_0 \right\Vert + \left\Vert \A^{5/2} \varphi_1 \right\Vert \right)\nonumber \\
				&+ \tau^4 \frac{t_k}{8} \left( \frac{3}{5} \tau + \frac{55}{18} t_k \right) \sup_{0 \leq t \leq t_k} \left\Vert \A^{5/2} f \left( t \right) \right\Vert\nonumber \\
				&+ \tau^4 \left( \frac{49}{1440} t_k \right) \sum_{s = 0}^{4} \binom{4}{s} \sup_{0 \leq t \leq t_k} \left\Vert \A^{2 - s / 2} f^{\left( s \right)} \left( t \right) \right\Vert\,,
			\end{align}
			for $k = 1,2,\ldots,n$.
		\end{theorem}		
		\begin{proof}
			By applying $\A^{1 / 2}$ to both sides of the equation \eqref{eq:Simps_u_tk}, the resulting system can be rewritten in matrix-vector form as follows:
			\begin{equation}\label{eq:matrix_vec_Simps_u_tk}
				\boldsymbol{w_k} = U \left( \tau \right) \boldsymbol{w_{k - 1}} + \frac{\tau}{6} \boldsymbol{g_k} + \boldsymbol{r_k} \left( \tau \right) \quad \text{for} \quad k = 1,2,\ldots,n\,,
			\end{equation}
			where $\boldsymbol{g_k} = \left( g_{1,k} \left( \tau \right),g_{2,k} \left( \tau \right) \right)^{\top}$ and $\boldsymbol{r_k} \left( \tau \right) = \left( \bar{R}_{1,k} \left( \tau \right),\bar{R}_{2,k} \left( \tau \right) \right)^{\top}$.
			
			The system of equations \eqref{eq:Simps_approx_u_tk} and \eqref{eq:Simps_diff_approx_u_tk}, in the same way as \eqref{eq:matrix_vec_Simps_u_tk}, may be reformulated as follows:
			\begin{equation}\label{eq:matrix_vec_Simps_approx_u_tk}
				\boldsymbol{v_k} = V \left( \tau \right) \boldsymbol{v_{k - 1}} + \frac{\tau}{6} \boldsymbol{\tilde{g}_k} \quad \text{for} \quad k = 1,2,\ldots,n\,,
			\end{equation}
			in which $\boldsymbol{\tilde{g}_k} = \left( \tilde{g}_{1,k} \left( \tau \right),\tilde{g}_{2,k} \left( \tau \right) \right)^{\top}$.
			
			By unrolling the recurrence relations \eqref{eq:matrix_vec_Simps_u_tk} and \eqref{eq:matrix_vec_Simps_approx_u_tk}, we respectively obtain:
			\begin{align}
				\boldsymbol{w_k} &= U \left( t_k \right) \boldsymbol{w_0} + \frac{\tau}{6} \sum_{i = 1}^{k} U \left( t_{k - i} \right) \boldsymbol{g_i} + \sum_{i = 1}^{k} U \left( t_{k - i} \right) \boldsymbol{r_i} \left( \tau \right)\,,\label{eq:unrolling_wk} \\
				\boldsymbol{v_k} &= \left( V \left( \tau \right) \right)^k \boldsymbol{v_0} + \frac{\tau}{6} \sum_{i = 1}^{k} \left( V \left( \tau \right) \right)^{k - i} \boldsymbol{\tilde{g}_i}\,.\label{eq:unrolling_vk}
			\end{align}
			
			If we subtract \eqref{eq:unrolling_vk} from \eqref{eq:unrolling_wk} and take into account that $\boldsymbol{w_0} = \boldsymbol{v_0} = \left( \A^{1 / 2} \varphi_0,\varphi_1 \right)^{\top}$, we obtain:
			\begin{align}\label{eq:norm_diff_w_v}
				\boldsymbol{w_k} - \boldsymbol{v_k} &= \left( U \left( t_k \right) - \left( V \left( \tau \right) \right)^k \right) \boldsymbol{w_0} + \frac{\tau}{6} \sum_{i = 1}^{k} U \left( t_{k - i} \right) \left( \boldsymbol{g_i} - \boldsymbol{\tilde{g}_i} \right)\nonumber \\
				&+ \frac{\tau}{6} \sum_{i = 1}^{k} \left( U \left( t_{k - i} \right) - \left( V \left( \tau \right) \right)^{k - i} \right) \boldsymbol{\tilde{g}_i} + \sum_{i = 1}^{k} U \left( t_{k - i} \right) \boldsymbol{r_i} \left( \tau \right)\,.
			\end{align}
			
			Let us estimate the norm of the difference $\boldsymbol{g_i} - \boldsymbol{\tilde{g}_i}$. Taking into account estimates \eqref{eq:cos_approx} and \eqref{eq:sin_approx}, we arrive at the following bound:
			\begin{align}\label{eq:diff_g_tilde_g}
				\left\Vert \boldsymbol{g_i} - \boldsymbol{\tilde{g}_i} \right\Vert_{\Hilb \times \Hilb} &\leq \left\Vert \left( \sin \left( \tau \A^{1 / 2} \right) - \W_2 \left( \tau \right) \right) f_{i - 1} \right\Vert + 4 \left\Vert \left( \sin \left( \frac{\tau}{2} \A^{1 / 2} \right) - \W_2 \left( \frac{\tau}{2} \right) \right) f_{i - 1/2} \right\Vert\nonumber \\
				&+ \left\Vert \left( \cos \left( \tau \A^{1 / 2} \right) - \W_1 \left( \tau \right) \right) f_{i - 1} \right\Vert + 4 \left\Vert \left( \cos \left( \frac{\tau}{2} \A^{1 / 2} \right) - \W_1 \left( \frac{\tau}{2} \right) \right) f_{i - 1/2} \right\Vert\nonumber \\
				&\leq \frac{287}{720} \tau^5 \left( \left\Vert \A^{5/2} f_{i - 1} \right\Vert + \frac{1}{8} \left\Vert \A^{5/2} f_{i - 1/2} \right\Vert \right) \leq \frac{287}{640} \tau^5 \sup_{0 \leq t \leq t_i} \left\Vert \A^{5/2} f \left( t \right) \right\Vert\,.
			\end{align}
			
			By \thmref{prop:thrm_unitary_aprox}, we have:
			\begin{equation}\label{eq:diff_Uki_Vki}
				\left\Vert \left( U \left( t_{k - i} \right) - \left( V \left( \tau \right) \right)^{k - i} \right) \boldsymbol{\tilde{g}_i} \right\Vert_{\Hilb \times \Hilb} \leq \tau^4 \left( \frac{5}{12} t_{k - i} \right) \left\Vert \left\{ \A^{5 / 2} \tilde{g}_{1,i},\A^{5/2} \tilde{g}_{2,i} \right\} \right\Vert_{\Hilb \times \Hilb}\,.
			\end{equation}
			
			Since $\left\Vert \W_1 \left( t \right) \right\Vert \leq 1$ and $\left\Vert \W_2 \left( t \right) \right\Vert \leq 1$, therefore the following inequality holds:
			\begin{equation}\label{eq:gi_norm_est}
				\left\Vert \left\{ \A^{5 / 2} \tilde{g}_{1,i},\A^{5/2} \tilde{g}_{2,i} \right\} \right\Vert_{\Hilb \times \Hilb} \leq \left\Vert \A^{5 / 2} \tilde{g}_{1,i} \right\Vert + \left\Vert \A^{5 / 2} \tilde{g}_{2,i} \right\Vert \leq 11 \sup_{0 \leq t \leq t_i} \left\Vert \A^{5/2} f \left( t \right) \right\Vert\,.
			\end{equation}
			Leveraging inequalities \eqref{eq:diff_Uki_Vki} and \eqref{eq:gi_norm_est}, we arrive at the following estimate:
			\begin{equation}\label{eq:sum_diff_Uki_Vki}
				\frac{\tau}{6} \left\Vert \sum_{i = 1}^{k} \left( U \left( t_{k - i} \right) - \left( V \left( \tau \right) \right)^{k - i} \right) \boldsymbol{\tilde{g}_i} \right\Vert_{\Hilb \times \Hilb} \leq \tau^4 \left( \frac{55}{144} t_k^2 \right) \sup_{0 \leq t \leq t_k} \left\Vert \A^{5/2} f \left( t \right) \right\Vert\,.
			\end{equation}
			
			Using inequality \eqref{eq:diff_U_V_global_general}, we find that
			\begin{equation}\label{eq:norm_diff_U_V}
				\left\Vert \left( U \left( t_k \right) - \left( V \left( \tau \right) \right)^k \right) \boldsymbol{w_0} \right\Vert_{\Hilb \times \Hilb} \leq \tau^4 \left( \frac{5}{12} t_k \right) \left( \left\Vert \A^3 \varphi_0 \right\Vert + \left\Vert \A^{5/2} \varphi_1 \right\Vert \right)\,.
			\end{equation}
			
			From \eqref{eq:Simps_error_unified}, we have
			\begin{equation}\label{eq:sum_remainders}
				\left\Vert \sum_{i = 1}^{k} U \left( t_{k - i} \right) \boldsymbol{r_i} \left( \tau \right) \right\Vert_{\Hilb \times \Hilb} \leq \tau^4 \left( \frac{49}{1440} t_k \right) \sum_{s = 0}^{4} \binom{4}{s} \sup_{0 \leq t \leq t_k} \left\Vert \A^{2 - s / 2} f^{\left( s \right)} \left( t \right) \right\Vert
			\end{equation}
			
			If we apply norms to both sides of the equality \eqref{eq:norm_diff_w_v}, take into consideration the inequalities \eqref{eq:diff_g_tilde_g}, \eqref{eq:sum_diff_Uki_Vki}, \eqref{eq:norm_diff_U_V}, and \eqref{eq:sum_remainders}, we derive estimate \eqref{eq:approx_error_diff_w_v}.
		\end{proof}
		
		We note that, for the homogeneous problem, the error estimate \eqref{eq:approx_error_diff_w_v} corresponding to scheme \eqref{eq:Simps_approx_u_tk}-\eqref{eq:Simps_diff_approx_u_tk} was established in \cite{LomtadzeRogavaTsiklauri2004}.
		
		We now formulate a theorem that establishes an error estimate for scheme \eqref{eq:Simps_approx_u_tk}-\eqref{eq:Simps_diff_approx_u_tk} in terms of the regularity of the solution.
		
		\begin{theorem}\label{prop:thm_variable_regularity}
			Suppose that for $j = 1,2,3$, the data of problem  \eqref{eq:abstr_lin_hiperb}-\eqref{eq:init_data_lin_hiperb} satisfy the following regularity assumptions, namely: $\varphi_0 \in D \left( \A^{1 + j/2} \right)$, $\varphi_1 \in  D \left( \A^{\left( 1 + j \right)/2} \right)$, $f \left( t \right) \in  D \left( \A^{\left( 1 + j \right)/2} \right)$, and $f^{\left( i \right)} \left( t \right) \in  D \left( \A^{\left( j - i \right)/2} \right)$ with $i = 1,\ldots,j$ (it is assumed that $D \left( \A^0 \right) = \Hilb$). Then the following estimates hold:
			\begin{align*}
				\left\Vert \boldsymbol{w_k} - \boldsymbol{v_k} \right\Vert_{\Hilb \times \Hilb} &\leq \tau^j c_{5 + 5j} t_k \left( \left\Vert \A^{1 + j/2} \varphi_0 \right\Vert + \left\Vert \A^{\left( 1 + j \right)/2} \varphi_1 \right\Vert \right) \\
				&+ \tau^j c_{6 + 5j} t_k \left( c_{7 + 5j} \tau + c_{8 + 5j} t_k \right) \sup_{0 \leq t \leq t_k} \left\Vert \A^{\left( 1 + j \right)/2} f \left( t \right) \right\Vert \\
				&+ \tau^j c_{9 + 5j} t_k \sum_{s = 0}^{j} \binom{j}{s} \sup_{0 \leq t \leq t_k} \left\Vert \A^{\left( j - s \right)/2} f^{\left( s \right)} \left( t \right) \right\Vert\,.
			\end{align*}
			Here, the approximate solution $\boldsymbol{v_k}$ is computed using the scheme \eqref{eq:Simps_approx_u_tk}-\eqref{eq:Simps_diff_approx_u_tk} for $j = 3$. When $j = 2$ and $j = 1$, $\boldsymbol{v_k}$ is defined by a similar scheme, obtained by approximating the integral terms appearing in the exact representation \eqref{eq:Simps_u_tk}-\eqref{eq:Simps_diff_u_tk} using the trapezoidal rule. In this case, the factor $\tau / 6$ in scheme \eqref{eq:Simps_approx_u_tk}-\eqref{eq:Simps_diff_approx_u_tk} is replaced by $\tau / 2$. Moreover, the vectors $\tilde{g}_{1,k}$ and $\tilde{g}_{2,k}$ are given by
			\begin{equation*}
				\tilde{g}_{1,k} \left( \tau \right) = \W_2 \left( \tau \right) f_{k - 1} \quad \text{and} \quad \tilde{g}_{2,k} \left( \tau \right) = \W_1 \left( \tau \right) f_{k - 1} + f_k\,.
			\end{equation*}
			
			The constants involved in the above estimate are absolute constants, and their values are as follows: $\left( c_{10}, c_{11}, c_{12}, c_{13}, c_{14} \right) = \left( \frac{11}{3}, \frac{1}{1}, \frac{5}{3}, \frac{11}{4}, \frac{2}{1} \right)$, $\left( c_{15}, c_{16}, c_{17}, c_{18}, c_{19} \right) = \left( \frac{7}{3}, \frac{1}{4}, \frac{13}{3}, \frac{7}{1}, \frac{5}{6} \right)$, and $\left( c_{20}, c_{21}, c_{22}, c_{23}, c_{24} \right) = \left( \frac{5}{4}, \frac{5}{16}, \frac{3}{4}, \frac{11}{3}, \frac{1}{6} \right)$.
		\end{theorem}
		\begin{proof}
			The proof proceeds analogously to that of \thmref{prop:thm_convrg_w_v}.
		\end{proof}
		
		So far, in order to determine the order of convergence of the approximate solution, we have required the data of the continuous problem to possess greater smoothness than is necessary for the existence of a classical solution. It is well known that if $\varphi_0 \in D \left( \A \right)$, $\varphi_1 \in D \left( \A^{1/2} \right)$, and either $f \left( t \right)$ is continuously differentiable, or $f \left( t \right) \in D \left( \A^{1/2} \right)$ for every fixed $t \in \left[ 0,T \right]$ and $\A^{1/2} f \left( t \right)$ is continuous on $\left[ 0,T \right]$, then problem \eqref{eq:abstr_lin_hiperb}–\eqref{eq:init_data_lin_hiperb} admits a unique solution (see Kre\u{\i}n \cite{Krein1971}). A natural question, therefore, arises: under these assumptions, is it possible to determine the order of approximation of the approximate solution obtained by the proposed method?
		
		In this regard, we state below, as a remark, a fact that follows from representation \eqref{eq:norm_diff_w_v}, after replacing $\tau / 6$ by $\tau$ and setting $\boldsymbol{g_k} = \boldsymbol{\tilde{g}_k}= \left( 0,f_k \right)^{\top}$, provided that estimate \eqref{eq:U_minus_V_general} is taken into account and certain standard manipulations are carried out in estimating the remainder term of the right-rectangle quadrature rule for the integral terms occurring in representations \eqref{eq:exact_u_tk} and \eqref{eq:exact_diff_u_tk}.
		\begin{remark}
			Under the above assumptions on $\varphi_0$, $\varphi_1$, and $f \left( t \right)$, the following estimate holds:
			\begin{align*}
				\left\Vert \boldsymbol{w_k} - \boldsymbol{v_k} \right\Vert_{\Hilb \times \Hilb} &\leq \tau^{1/8} a \left( t_k \right) \left( \left\Vert \A \varphi_0 \right\Vert + \left\Vert \A^{1/2} \varphi_1 \right\Vert \right) \\
				&+ \tau^{1/8} t_k \left( a \left( t_k \right) + \tau^{7/8} \right) \sup_{0 \leq t \leq t_k} \left\Vert \A^{1/2} f \left( t \right) \right\Vert + \frac{1}{2} \tau t_k \sup_{0 \leq t \leq t_k} \left\Vert f^{\prime} \left( t \right) \right\Vert\,.
			\end{align*}
		\end{remark}
		
		\subsection{Semilinear Problem}
		
		Let us consider a Cauchy problem for an abstract hyperbolic semilinear equation in Hilbert space $\Hilb$:
		\begin{gather}
			\frac{\d^2 u \left( t \right)}{\d t^2} + \A u \left( t \right) = f \left( t \right) + M \left( u \left( t \right), u^{\prime} \left( t \right), \B u \left( t \right) \right)\,, \quad t \in \left[ 0,T \right]\,,\label{eq:abstr_semi-lin_hiperb} \\
			u \left( 0 \right) = \varphi_0\,, \quad \left.\frac{\d u \left( t \right)}{\d t} \right|_{t = 0} = \varphi_1\,.\label{eq:init_data_semi-lin_hiperb}
		\end{gather}
		
		In this context, the operator $\A$ satisfies the same assumptions as those stated in \hyperref[subsec:lin_probl]{Subsection \ref*{subsec:lin_probl}}. The same applies to the initial conditions and the function $f \left( t \right)$. The operator $\B$ is a closed linear operator satisfying the subordination condition $D \left( \A \right) \subset D \left( \B \right)$, and
		\begin{equation}\label{eq:subord_cond}
			\left\Vert \B \varphi \right\Vert^2 \leq b^2 \left( \A \varphi,\varphi \right)\,,
		\end{equation}
		for all $\varphi \in D \left( \A \right)$ and for some constant $b > 0$. The nonlinear operator $M : \Hilb^3 \to \Hilb$ is assumed to be Lipschitz continuous; \ie{},
		\begin{equation}\label{eq:semilin_lip_cond}
			\left\Vert M \left( x_1, y_1, z_1 \right) - M \left( x_2, y_2, z_2 \right) \right\Vert \leq \alpha_1 \left( \left\Vert x_1 - x_2 \right\Vert + \left\Vert y_1 - y_2 \right\Vert + \left\Vert z_1 - z_2 \right\Vert \right)\,, \quad \alpha_1 > 0\,,
		\end{equation}
		for all $\left( x_j, y_j, z_j \right) \in \Hilb^3$, with $j = 1,2$.
		
		\begin{remark}\label{prop:domain_exst}
			From condition \eqref{eq:subord_cond}, the following relation is deduced:
			\begin{equation}\label{eq:domain_exst_final_res}
				\left\Vert \B \varphi \right\Vert \leq b \left\Vert \A^{1/2} \varphi \right\Vert\,,\quad \forall \varphi \in D \left( \A^{1/2} \right) \subset D \left( \B \right)\,.
			\end{equation}
		\end{remark}
		\begin{proof}
			It is evident that from \eqref{eq:subord_cond}, the following implication arises:
			\begin{equation}\label{eq:domain_exst_relat_norm_B_norm_halfA}
				\left\Vert \B \varphi \right\Vert \leq b \left\Vert \A^{1/2} \varphi \right\Vert\,,\quad \forall \varphi \in D \left( \A \right) \subset D \left( \B \right)\,.
			\end{equation}
			It is known that $D \left( \A \right)$ is a core of $\A^{1/2}$ (see, \textbf{Lemma 3.38} in Kato \cite{Kato1980}). This implies that for every $\varphi \in D \left( \A^{1/2} \right)$, there exists a sequence $\varphi_n \in D \left( \A \right)$ such that $\varphi_n \to \varphi$ and $\A^{1/2} \varphi_n \to \A^{1/2} \varphi$. Consequently, by \eqref{eq:domain_exst_relat_norm_B_norm_halfA}, it follows that $\B \varphi_n$ forms a Cauchy sequence. Given that $\Hilb$ is a complete space, it is evident that this sequence converges. Furthermore, since the operator $\B$ is closed, we conclude that $\varphi \in D \left( \B \right)$ and $\B \varphi_n \to \B \varphi$. As a result of this reasoning and inequality \eqref{eq:domain_exst_relat_norm_B_norm_halfA}, the relation \eqref{eq:domain_exst_final_res} is established.
		\end{proof}
		
		If we express the solution of problem \eqref{eq:abstr_semi-lin_hiperb}–\eqref{eq:init_data_semi-lin_hiperb} using formulas \eqref{eq:exact_u_tk} and \eqref{eq:exact_diff_u_tk}, and then apply the right-rectangle quadrature rule to the integral appearing in the representation of $u \left( t_k \right)$ and the left-rectangle quadrature rule to the integral occurring in the representation of $u^{\prime} \left( t_k \right)$, we obtain:
		\begin{align}
			u \left( t_k \right) &= \cos\left( \tau \A^{1 / 2} \right) u \left( t_{k - 1} \right) + \sin\left( \tau \A^{1 / 2} \right) \A^{-1 / 2} u^{\prime} \left( t_{k - 1} \right)\nonumber \\
			&+ \tau \A^{-1 / 2} g_{1,k} \left( \tau \right) + \A^{-1 / 2} \bar{R}_{1,k} \left( \tau \right)\,,\label{eq:semilin_rect_u_tk} \\
			u^{\prime} \left( t_k \right) &= - \sin\left( \tau \A^{1 / 2} \right) \A^{1 / 2} u \left( t_{k - 1} \right) + \cos\left( \tau \A^{1 / 2} \right) u^{\prime} \left( t_{k - 1} \right)\nonumber \\
			&+ \tau g_{2,k} \left( \tau \right) + \bar{R}_{2,k} \left( \tau \right)\,,\label{eq:semilin_rect_diff_u_tk}
		\end{align}
		where $g_{1,k} \left( \tau \right) = 0$ and
		\begin{equation*}
			g_{2,k} \left( \tau \right) = \cos\left( \tau \A^{1 / 2} \right) \left[ f \left( t_{k - 1} \right) + \F \left( u \left( t_{k - 1} \right) \right) \right]\,, \quad \text{with} \quad \F \left( u \left( t \right) \right) = M \left( u \left( t \right), u^{\prime} \left( t \right), \B u \left( t \right) \right)\,.
		\end{equation*}
		The quantities $\bar{R}_{1,k} \left( \tau \right)$ and $\bar{R}_{2,k} \left( \tau \right)$ denote the remainder terms associated with the corresponding rectangle quadrature rules.
		
		Assume that the following conditions are fulfilled:
		\begin{enumerate}[label=(\alph*)]
			\item\label{item:semilin_a} The function $f \left( t \right)$ is continuously differentiable on $\left[ 0,T \right]$. Furthermore, $f \left( t \right) \in D \left( \A^{1/2} \right)$ for each $t \in \left[ 0,T \right]$ and $\A^{1/2} f \left( t \right)$ is continuous.
			\item\label{item:semilin_b} The function $u \left( t \right)$ is continuous on the interval $\left[ 0,T \right]$ and twice continuously differentiable, with $u^{\prime} \left( t \right) \in D \left( \A^{1/2} \right)$ for each $t \in \left[ 0,T \right]$.
			\item\label{item:semilin_c} The nonlinear operator $M : \Hilb^3 \to \Hilb$ satisfies condition \eqref{eq:semilin_lip_cond} (the Lipschitz condition). Moreover, $\F \left( u \left( t \right) \right) \in D \left( \A^{1/2} \right)$ for each $t \in \left[ 0,T \right]$ and the function $\A^{1/2} \F \left( u \left( t \right) \right)$  is continuous.
		\end{enumerate}
		
		Under these assumptions, the following estimates are valid for the remainders $\bar{R}_{1,k} \left( \tau \right)$ and $\bar{R}_{2,k} \left( \tau \right)$:
		\begin{gather}
			\left\Vert  \bar{R}_{1,k} \left( \tau \right) \right\Vert \leq \frac{\tau^2}{2} J_k \left( f,\F \right)\,,\label{eq:semilin_rem1} \\
			\left\Vert  \bar{R}_{2,k} \left( \tau \right) \right\Vert \leq \frac{\tau^2}{2} \alpha_2 \left( J_k \left( f,\F \right) + \left\Vert f^{\prime} \right\Vert_{C \left( I_k \right)} + \left\Vert u^{\prime} \right\Vert_{C \left( I_k \right)} + \left\Vert u^{\prime \prime} \right\Vert_{C \left( I_k \right)} + b \left\Vert \A^{1 / 2} u^{\prime} \right\Vert_{C \left( I_k \right)} \right)\,,\label{eq:semilin_rem2}
		\end{gather}
		where $\left\Vert u \right\Vert_{C \left( I_k \right)} = \sup\limits_{t \in I_k} \left\Vert u \left( t \right) \right\Vert$ with $I_k = \left[ 0,t_k \right]$, $\alpha_2 = \max \left( 1, \alpha_1 \right)$, and
		\begin{equation*}
			J_k \left( f,\F \right) = \left\Vert \A^{1 / 2} f \right\Vert_{C \left( I_k \right)} + \left\Vert \A^{1 / 2} \F \left( u \right) \right\Vert_{C \left( I_k \right)}\,.
		\end{equation*}
		
		Replacing, in the representations \eqref{eq:semilin_rect_u_tk} and \eqref{eq:semilin_rect_diff_u_tk}, the cosine and sine operator functions by their rational approximations $\W_1 \left( \tau \right)$ and $\W_2 \left( \tau \right)$, respectively, and omitting the remainder terms $\bar{R}_{1,k} \left( \tau \right)$ and $\bar{R}_{2,k} \left( \tau \right)$, yields the following scheme:
		\begin{align}
			u_k^{\left( 0 \right)} &= \W_1 \left( \tau \right) u_{k - 1}^{\left( 0 \right)} + \W_2 \left( \tau \right) \A^{-1 / 2} u_{k - 1}^{\left( 1 \right)} + \tau \A^{-1 / 2} \tilde{g}_{1,k}\,,\label{eq:semilin_rect_u_k0} \\
			u_k^{\left( 1 \right)} &= - \W_2 \left( \tau \right) \A^{1 / 2} u_{k - 1}^{\left( 0 \right)} + \W_1 \left( \tau \right) u_{k - 1}^{\left( 1 \right)} + \tau \tilde{g}_{2,k}\,,\label{eq:semilin_rect_u_k1}
		\end{align}
		where $k = 1,2,\ldots,n$, $u_0^{\left( 0 \right)} = \varphi_0$, $u_0^{\left( 1 \right)} = \varphi_1$, and
		\begin{equation*}
			\tilde{g}_{1,k} \left( \tau \right) = 0\,, \quad \tilde{g}_{2,k} \left( \tau \right) = \W_1 \left( \tau \right) \left( f_{k - 1} + \widetilde{\F}_{k - 1} \right)\,, \quad \text{with} \quad \widetilde{\F}_k = M \left( u_k^{\left( 0 \right)}, u_k^{\left( 1 \right)}, \B u_k^{\left( 0 \right)} \right)\,.
		\end{equation*}
		
		The following theorem holds.
		\begin{theorem}\label{prop:semilin_main_thm}
			Let $\varphi_0 \in D \left( \A \right)$ and $\varphi_1 \in D \left( \A^{1/2} \right)$. Moreover, suppose that conditions \hyperref[item:semilin_a]{\ref{item:semilin_a}}, \hyperref[item:semilin_b]{\ref{item:semilin_b}}, and \hyperref[item:semilin_c]{\ref{item:semilin_c}} are fulfilled. Then the error $\boldsymbol{w_k} - \boldsymbol{v_k}$ of the scheme \eqref{eq:semilin_rect_u_k0}–\eqref{eq:semilin_rect_u_k1} satisfies the following estimate:
			\begin{equation*}
				\left\Vert \boldsymbol{w_k} - \boldsymbol{v_k} \right\Vert_{\Hilb \times \Hilb} \leq \hat{c}_k \left( \tau \right) e^{\mu t_k} \leq  \tau^{1/8} b_0  e^{\mu t_k}\,, \quad k = 1,2,\ldots,n\,,
			\end{equation*}
			where $b_0 = \mathrm{const} > 0$, $\mu = \alpha_1 b_1$, and $b_1 = \sqrt{2} \max \left( 1,\alpha^{-1/2} + b \right)$.
			
			Here, the coefficient $\hat{c}_k \left( \tau \right)$ is of order $\bigO \left( \tau^{1/8} \right)$ and depends linearly on $\left\Vert \A \varphi_0 \right\Vert$, $\left\Vert \A^{1/2} \varphi_1 \right\Vert$, $\left\Vert \A^{1 / 2} f \right\Vert_{C \left( I_k \right)}$, $\left\Vert \A^{1 / 2} \F \left( u \right) \right\Vert_{C \left( I_k \right)}$, $\left\Vert f^{\prime} \right\Vert_{C \left( I_k \right)}$, $\left\Vert u^{\prime} \right\Vert_{C \left( I_k \right)}$, $\left\Vert u^{\prime \prime} \right\Vert_{C \left( I_k \right)}$, and $\left\Vert \A^{1 / 2} u^{\prime} \right\Vert_{C \left( I_k \right)}$.
		\end{theorem}
		
		\begin{proof}
			Similarly to \eqref{eq:norm_diff_w_v}, the following representation is valid for the error of scheme \eqref{eq:semilin_rect_u_k0}–\eqref{eq:semilin_rect_u_k1}:
			\begin{align}\label{eq:semilin_error_eqt}
				\boldsymbol{w_k} - \boldsymbol{v_k} &= \left( U \left( t_k \right) - \left( V \left( \tau \right) \right)^k \right) \boldsymbol{w_0} + \tau \sum_{i = 1}^{k} \left( V \left( \tau \right) \right)^{k - i} \left( \boldsymbol{g_i} - \boldsymbol{\tilde{g}_i} \right)\nonumber \\
				&+ \tau \sum_{i = 1}^{k} \left( U \left( t_{k - i} \right) - \left( V \left( \tau \right) \right)^{k - i} \right) \boldsymbol{g_i} + \sum_{i = 1}^{k} U \left( t_{k - i} \right) \boldsymbol{r_i} \left( \tau \right)\,.
			\end{align}
			
			By virtue of \eqref{eq:U_minus_V_general}, we obtain:
			\begin{gather}
				\left\Vert \left( U \left( t_k \right) - \left( V \left( \tau \right) \right)^k \right) \boldsymbol{w_0} \right\Vert_{\Hilb \times \Hilb} \leq a \left( t_k \right) \tau^{1/8} \left( \left\Vert \A \varphi_0 \right\Vert + \left\Vert \A^{1/2} \varphi_1 \right\Vert \right)\,,\label{eq:semilin_est_w0} \\
				\left\Vert \left( U \left( t_{k - i} \right) - \left( V \left( \tau \right) \right)^{k - i} \right) \boldsymbol{g_i} \right\Vert_{\Hilb \times \Hilb} \leq a \left( t_k \right) \tau^{1/8} J_k \left( f,\F \right)\,.\label{eq:semilin_est_gi}
			\end{gather}
			It follows from \eqref{eq:semilin_est_gi} that
			\begin{equation}\label{eq:semilin_sum_gi}
				\tau \left\Vert \sum_{i = 1}^{k} \left( U \left( t_{k - i} \right) - \left( V \left( \tau \right) \right)^{k - i} \right) \boldsymbol{g_i} \right\Vert_{\Hilb \times \Hilb} \leq \tau^{1/8} t_k a \left( t_k \right) J_k \left( f,\F \right)\,.
			\end{equation}
			
			Combining estimates \eqref{eq:semilin_rem1} and \eqref{eq:semilin_rem2}, we arrive at
			\begin{align}\label{eq:semilin_sum_rem_r}
				&\left\Vert \sum_{i = 1}^{k} U \left( t_{k - i} \right) \boldsymbol{r_i} \left( \tau \right) \right\Vert_{\Hilb \times \Hilb} \leq\nonumber \\ &\frac{\tau}{2} \alpha_2 t_k \left( 2 J_k \left( f,\F \right) + \left\Vert f^{\prime} \right\Vert_{C \left( I_k \right)} + \left\Vert u^{\prime} \right\Vert_{C \left( I_k \right)} + \left\Vert u^{\prime \prime} \right\Vert_{C \left( I_k \right)} + b \left\Vert \A^{1 / 2} u^{\prime} \right\Vert_{C \left( I_k \right)} \right)\,.
			\end{align}
			
			It remains to estimate the norm of the vector $\boldsymbol{g_i} - \boldsymbol{\tilde{g}_i}$. Taking into account estimate \eqref{eq:est_cosine_minus_W} and condition \eqref{eq:semilin_lip_cond}, we obtain:
			\begin{align}\label{eq:semilin_diff_gi_tild_gi}
				\left\Vert \boldsymbol{g_i} - \boldsymbol{\tilde{g}_i} \right\Vert_{\Hilb \times \Hilb} &\leq \left\Vert \left( \cos \left( \tau \A^{1 / 2} \right) - \W_1 \left( \tau \right) \right) \left( f \left( t_{i - 1} \right) + \F \left( u \left( t_{i - 1} \right) \right) \right) \right\Vert\nonumber \\
				&+ \left\Vert \F \left( u \left( t_{i - 1} \right) \right) - \widetilde{\F}_{i - 1} \right\Vert \leq \frac{3}{2} \tau J_k \left( f,\F \right)\nonumber \\
				&+ \alpha_1 \left( \left\Vert u \left( t_{i - 1} \right) - u_{i - 1}^{\left( 0 \right)} \right\Vert + \left\Vert u^{\prime} \left( t_{i - 1} \right) - u_{i - 1}^{\left( 1 \right)} \right\Vert + \left\Vert \B u \left( t_{i - 1} \right) - \B u_{i - 1}^{\left( 0 \right)} \right\Vert \right)\,.
			\end{align}
			
			By the subordination condition \eqref{eq:subord_cond} and the positive definiteness of the operator $\A$, it follows that
			\begin{align}\label{eq:semilin_diff_wv_prev}
				&\left\Vert u \left( t_{i - 1} \right) - u_{i - 1}^{\left( 0 \right)} \right\Vert + \left\Vert u^{\prime} \left( t_{i - 1} \right) - u_{i - 1}^{\left( 1 \right)} \right\Vert + \left\Vert \B u \left( t_{i - 1} \right) - \B u_{i - 1}^{\left( 0 \right)} \right\Vert \leq\nonumber \\
				&\left\Vert u^{\prime} \left( t_{i - 1} \right) - u_{i - 1}^{\left( 1 \right)} \right\Vert + \left( \frac{1}{\sqrt{\alpha}} + b \right) \left\Vert \A^{1/2} u \left( t_{i - 1} \right) - \A^{1/2} u_{i - 1}^{\left( 0 \right)} \right\Vert \leq b_1 \left\Vert \boldsymbol{w_{i - 1}} - \boldsymbol{v_{i - 1}} \right\Vert_{\Hilb \times \Hilb}\,.
			\end{align}
			
			Combining inequalities \eqref{eq:semilin_diff_gi_tild_gi} and \eqref{eq:semilin_diff_wv_prev} with the estimate $\left\Vert V \left( \tau \right) \right\Vert \leq 1$, we arrive at the following inequality:
			\begin{equation}\label{eq:semilin_sum_diff_gi_til_gi}
				\tau \left\Vert \sum_{i = 1}^{k} \left( V \left( \tau \right) \right)^{k - i} \left( \boldsymbol{g_i} - \boldsymbol{\tilde{g}_i} \right) \right\Vert_{\Hilb \times \Hilb} \leq \frac{3}{2} \tau t_k J_k \left( f,\F \right) + \tau \alpha_1 b_1 \sum_{i = 0}^{k - 1} \left\Vert \boldsymbol{w_i} - \boldsymbol{v_i} \right\Vert_{\Hilb \times \Hilb}\,.
			\end{equation}
			
			From equality \eqref{eq:semilin_error_eqt}, taking into account inequalities \eqref{eq:semilin_est_w0}, \eqref{eq:semilin_sum_gi}, \eqref{eq:semilin_sum_rem_r}, and \eqref{eq:semilin_sum_diff_gi_til_gi}, we obtain the following inequality:
			\begin{equation}\label{eq:semilin_discrete_gronwall}
				\left\Vert \boldsymbol{w_k} - \boldsymbol{v_k} \right\Vert_{\Hilb \times \Hilb} \leq \hat{c}_k \left( \tau \right) + \tau \alpha_1 b_1 \sum_{i = 0}^{k - 1} \left\Vert \boldsymbol{w_i} - \boldsymbol{v_i} \right\Vert_{\Hilb \times \Hilb}\,.
			\end{equation}
			Here, $\hat{c}_k \left( \tau \right)$ denotes the sum of the right-hand sides of inequalities \eqref{eq:semilin_est_w0}, \eqref{eq:semilin_sum_gi}, and \eqref{eq:semilin_sum_rem_r}, together with the first term on the right-hand side of inequality \eqref{eq:semilin_sum_diff_gi_til_gi}.
			
			From \eqref{eq:semilin_discrete_gronwall}, by induction, one implies:
			\begin{equation*}
				\left\Vert \boldsymbol{w_k} - \boldsymbol{v_k} \right\Vert_{\Hilb \times \Hilb} \leq \left( 1 + \mu \tau \right)^{k - 1} \left( \hat{c}_k \left( \tau \right) + \mu \tau \varepsilon_0 \right) \leq e^{\mu t_k} \left( \hat{c}_k \left( \tau \right) + \mu \tau \varepsilon_0 \right)\,,
			\end{equation*}
			where $\varepsilon_0 = \left\Vert \boldsymbol{w_0} - \boldsymbol{v_0} \right\Vert_{\Hilb \times \Hilb}$ and $\mu = \alpha_1 b_1$.
			
			This inequality provides the desired estimate.
		\end{proof}
		
		The significance of \thmref{prop:semilin_main_thm} is that it determines the order of approximation of the scheme \eqref{eq:semilin_rect_u_k0}–\eqref{eq:semilin_rect_u_k1} under almost minimal smoothness assumptions on the data of the problem. If the smoothness order of the data of problem \eqref{eq:abstr_semi-lin_hiperb}–\eqref{eq:init_data_semi-lin_hiperb} is increased, roughly speaking, by $1/2$, then the approximation order of scheme \eqref{eq:semilin_rect_u_k0}–\eqref{eq:semilin_rect_u_k1} becomes $\bigO \left( \tau \right)$. More precisely, we have the following theorem.
		\begin{theorem}
			Let $\varphi_0 \in D \left( \A^{3/2} \right)$ and $\varphi_1 \in D \left( \A \right)$. Moreover, suppose that conditions \hyperref[item:semilin_a]{\ref{item:semilin_a}} and \hyperref[item:semilin_c]{\ref{item:semilin_c}} are fulfilled with the operator $\A^{1/2}$ replaced by $\A$, while condition \hyperref[item:semilin_b]{\ref{item:semilin_b}} remains unchanged. Then, for the error of scheme \eqref{eq:semilin_rect_u_k0}–\eqref{eq:semilin_rect_u_k1}, the following estimate holds:
			\begin{equation*}
				\left\Vert \boldsymbol{w_k} - \boldsymbol{v_k} \right\Vert_{\Hilb \times \Hilb} \leq  \hat{b}_0 \tau\,, \quad k = 1,2,\ldots,n\,, \quad b_0 = \mathrm{const} > 0\,.
			\end{equation*}
		\end{theorem}
		\begin{proof}
			The proof proceeds analogously to that of \thmref{prop:semilin_main_thm}. The main difference lies in deriving estimates similar to \eqref{eq:semilin_est_w0} and \eqref{eq:semilin_est_gi}; these estimates rely on the estimate stated in \thmref{prop:thm_addit_order} for $i = 2$.
		\end{proof}
		
		\section*{Funding}
		
		\noindent The second author, Z.V., was supported by the Shota Rustaveli National Science Foundation of Georgia (SRNSFG) under grant number FR-25-215.
		
		\section*{Conflicts of Interest}
		
		\noindent The authors declare no conflicts of interest.

		% ===================== BIBLIOGRAPHY OUTPUT LOCATION ==========================
		% --- BIBLIOGRAPHY -----------------------------------------------------------
		% References section with visible heading + PDF bookmark
		% Requires: \usepackage{biblatex} (with \addbibresource{...}) and \usepackage{hyperref}
		
		% % === METHOD: Start bibliography on a clean page (auto oneside/twoside) ===
		% \makeatletter                                   % enable access to class internals (@ macros)
		% \if@twoside\cleardoublepage\else\clearpage\fi   % right-hand page for twoside, plain new page for oneside
		% \makeatother                                    % restore normal category code for @
		
		% --- METHOD: Ensure correct hyperlink anchor --------------------------------
		\phantomsection                % Ensure correct hyperlink anchor for references (used with hyperref)
		
		% === METHOD: Print Bibliography with visible heading =======================
		% BibLaTeX creates a section-like heading ("References") and prints entries.
		% With hyperref loaded, this normally adds one PDF bookmark automatically.
		\printbibliography[
		heading=bibliography,   % use biblatex’s built-in section-level heading
		title={References}      % visible heading text
		]                         % no manual bookmark needed in most setups

@book {Kato1980,
	AUTHOR = {Kato, T.},
	TITLE = {Perturbation theory for linear operators},
	SERIES = {Classics in Mathematics},
	EDITION = {1980},
	PUBLISHER = {Springer-Verlag, Berlin},
	YEAR = {1995},
	PAGES = {xxii+619},
	ISBN = {3-540-58661-X},
	MRCLASS = {47A55 (46-00 47-00)},
	MRNUMBER = {1335452},
}

@article {Stone1932,
	AUTHOR = {Stone, M. H.},
	TITLE = {On one-parameter unitary groups in {H}ilbert space},
	JOURNAL = {Ann. of Math. (2)},
	FJOURNAL = {Annals of Mathematics. Second Series},
	VOLUME = {33},
	YEAR = {1932},
	NUMBER = {3},
	PAGES = {643--648},
	ISSN = {0003-486X,1939-8980},
	MRCLASS = {DML},
	MRNUMBER = {1503079},
	DOI = {10.2307/1968538},
	URL = {https://doi.org/10.2307/1968538},
}

@book {Yosida1980,
	AUTHOR = {Yosida, K.},
	TITLE = {Functional analysis},
	SERIES = {Classics in Mathematics},
	NOTE = {Reprint of the sixth (1980) edition},
	PUBLISHER = {Springer-Verlag, Berlin},
	YEAR = {1995},
	PAGES = {xii+501},
	ISBN = {3-540-58654-7},
	MRCLASS = {46-01 (47-01)},
	MRNUMBER = {1336382},
	DOI = {10.1007/978-3-642-61859-8},
	URL = {https://doi.org/10.1007/978-3-642-61859-8},
}

@book {ReedSimon1980,
	AUTHOR = {Reed, M. and Simon, B.},
	TITLE = {Methods of modern mathematical physics. {I}},
	EDITION = {Second},
	NOTE = {Functional analysis},
	PUBLISHER = {Academic Press, Inc. [Harcourt Brace Jovanovich, Publishers],
	New York},
	YEAR = {1980},
	PAGES = {xv+400},
	ISBN = {0-12-585050-6},
	MRCLASS = {46-01 (47-01)},
	MRNUMBER = {751959},
}

@book {Krein1971,
	AUTHOR = {Kre\u{\i}n, S. G.},
	TITLE = {Linear differential equations in {B}anach space},
	SERIES = {Translations of Mathematical Monographs, Vol. 29},
	NOTE = {Translated from the Russian by J. M. Danskin},
	PUBLISHER = {American Mathematical Society, Providence, RI},
	YEAR = {1971},
	PAGES = {v+390},
	MRCLASS = {34G05 (47E05 47F05)},
	MRNUMBER = {342804},
}

@article {RogavaOperSemigroup2022,
	AUTHOR = {Rogava, J. L.},
	TITLE = {Approximation of operator semigroups using linear-fractional
	operator functions and weighted averages},
	NOTE = {Translation of Funktsional. Anal. i Prilozhen. {\bf 56}
	(2022), no. 2, 47--63. },
	JOURNAL = {Funct. Anal. Appl.},
	FJOURNAL = {Functional Analysis and its Applications},
	VOLUME = {56},
	YEAR = {2022},
	NUMBER = {2},
	PAGES = {116--129},
	ISSN = {0016-2663,1573-8485},
	MRCLASS = {47A58 (41A30 47B65 47D06)},
	MRNUMBER = {4500302},
	MRREVIEWER = {Loredana-Florentina\ Iambor},
	DOI = {10.1134/s0016266322020058},
	URL = {https://doi.org/10.1134/s0016266322020058},
}

@article {Sova1966,
	AUTHOR = {Sova, M.},
	TITLE = {Cosine operator functions},
	JOURNAL = {Rozprawy Mat.},
	FJOURNAL = {Rozprawy Matematyczne},
	VOLUME = {49},
	YEAR = {1966},
	PAGES = {47},
	ISSN = {0860-2581},
	MRCLASS = {47.50},
	MRNUMBER = {193525},
	MRREVIEWER = {J.\ Gil de Lamadrid},
}

@article {Hoppe1982,
	AUTHOR = {Hoppe, R. H. W.},
	TITLE = {Discrete approximations of cosine operator functions. {I}},
	JOURNAL = {SIAM J. Numer. Anal.},
	FJOURNAL = {SIAM Journal on Numerical Analysis},
	VOLUME = {19},
	YEAR = {1982},
	NUMBER = {6},
	PAGES = {1110--1128},
	ISSN = {0036-1429},
	MRCLASS = {65J10 (34G10 47D05)},
	MRNUMBER = {679655},
	MRREVIEWER = {H.\ O.\ Fattorini},
	DOI = {10.1137/0719081},
	URL = {https://doi.org/10.1137/0719081},
}

@article {Goldstein1974,
	AUTHOR = {Goldstein, J. A.},
	TITLE = {On the convergence and approximation of cosine functions},
	JOURNAL = {Aequationes Math.},
	FJOURNAL = {Aequationes Mathematicae},
	VOLUME = {11},
	YEAR = {1974},
	PAGES = {201--205},
	ISSN = {0001-9054,1420-8903},
	MRCLASS = {47D99},
	MRNUMBER = {358435},
	MRREVIEWER = {G.\ Da-Prato},
	DOI = {10.1007/BF01832857},
	URL = {https://doi.org/10.1007/BF01832857},
}

@article {TravisWebb1978,
	AUTHOR = {Travis, C. C. and Webb, G. F.},
	TITLE = {Cosine families and abstract nonlinear second order
	differential equations},
	JOURNAL = {Acta Math. Acad. Sci. Hungar.},
	FJOURNAL = {Acta Mathematica. Academiae Scientiarum Hungaricae},
	VOLUME = {32},
	YEAR = {1978},
	NUMBER = {1-2},
	PAGES = {75--96},
	ISSN = {0001-5954,1588-2632},
	MRCLASS = {34G05 (47D05)},
	MRNUMBER = {499581},
	MRREVIEWER = {Vadim\ Komkov},
	DOI = {10.1007/BF01902205},
	URL = {https://doi.org/10.1007/BF01902205},
}

@article {Baker1976,
	AUTHOR = {Baker, G. A.},
	TITLE = {Error estimates for finite element methods for second order
	hyperbolic equations},
	JOURNAL = {SIAM J. Numer. Anal.},
	FJOURNAL = {SIAM Journal on Numerical Analysis},
	VOLUME = {13},
	YEAR = {1976},
	NUMBER = {4},
	PAGES = {564--576},
	ISSN = {0036-1429},
	MRCLASS = {65N30},
	MRNUMBER = {423836},
	MRREVIEWER = {Philip\ Brenner},
	DOI = {10.1137/0713048},
	URL = {https://doi.org/10.1137/0713048},
}

@article {BakerBramble1979,
	AUTHOR = {Baker, G. A. and Bramble, J. H.},
	TITLE = {Semidiscrete and single step fully discrete approximations for
	second order hyperbolic equations},
	JOURNAL = {RAIRO Anal. Num\'{e}r.},
	FJOURNAL = {RAIRO Analyse Num\'{e}rique},
	VOLUME = {13},
	YEAR = {1979},
	NUMBER = {2},
	PAGES = {75--100},
	ISSN = {0399-0516,0516-2777},
	MRCLASS = {65N30},
	MRNUMBER = {533876},
	MRREVIEWER = {Alexander\ Doktor},
	DOI = {10.1051/m2an/1979130200751},
	URL = {https://doi.org/10.1051/m2an/1979130200751},
}

@article {BakerDougalisSerbin1979,
	AUTHOR = {Baker, G. A. and Dougalis, V. A. and Serbin, S. M.},
	TITLE = {High order accurate two-step approximations for hyperbolic
	equations},
	JOURNAL = {RAIRO Anal. Num\'{e}r.},
	FJOURNAL = {RAIRO Analyse Num\'{e}rique},
	VOLUME = {13},
	YEAR = {1979},
	NUMBER = {3},
	PAGES = {201--226},
	ISSN = {0399-0516,0516-2777},
	MRCLASS = {65M05},
	MRNUMBER = {543933},
	MRREVIEWER = {David\ Gottlieb},
	DOI = {10.1051/m2an/1979130302011},
	URL = {https://doi.org/10.1051/m2an/1979130302011},
}

@article {BakerDougalisSerbin1980,
	AUTHOR = {Baker, G. A. and Dougalis, V. A. and Serbin, S. M.},
	TITLE = {An approximation theorem for second-order evolution equations},
	JOURNAL = {Numer. Math.},
	FJOURNAL = {Numerische Mathematik},
	VOLUME = {35},
	YEAR = {1980},
	NUMBER = {2},
	PAGES = {127--142},
	ISSN = {0029-599X,0945-3245},
	MRCLASS = {34G10 (65J10 65L05)},
	MRNUMBER = {585242},
	MRREVIEWER = {V.\ G.\ Kolomiets},
	DOI = {10.1007/BF01396311},
	URL = {https://doi.org/10.1007/BF01396311},
}

@article {Bales1993,
	AUTHOR = {Bales, L. A.},
	TITLE = {Semidiscrete and single step fully discrete finite element
	approximations for second order hyperbolic equations with
	nonsmooth solutions},
	JOURNAL = {RAIRO Mod\'{e}l. Math. Anal. Num\'{e}r.},
	FJOURNAL = {RAIRO Mod\'{e}lisation Math\'{e}matique et Analyse
	Num\'{e}rique},
	VOLUME = {27},
	YEAR = {1993},
	NUMBER = {1},
	PAGES = {55--63},
	ISSN = {0764-583X},
	MRCLASS = {65M60},
	MRNUMBER = {1204628},
	MRREVIEWER = {Chuan\ Miao\ Chen},
	DOI = {10.1051/m2an/1993270100551},
	URL = {https://doi.org/10.1051/m2an/1993270100551},
}

@article {Kacur1984,
	AUTHOR = {Ka\v{c}ur, J.},
	TITLE = {Application of {R}othe's method to perturbed linear hyperbolic
	equations and variational inequalities},
	JOURNAL = {Czechoslovak Math. J.},
	FJOURNAL = {Czechoslovak Mathematical Journal},
	VOLUME = {34(109)},
	YEAR = {1984},
	NUMBER = {1},
	PAGES = {92--106},
	ISSN = {0011-4642},
	MRCLASS = {35L75 (34G20 47H99)},
	MRNUMBER = {731982},
	MRREVIEWER = {Sadakazu\ Aizawa},
}

@article {Pultar1984,
	AUTHOR = {Pultar, M.},
	TITLE = {Solutions of abstract hyperbolic equations by {R}othe method},
	JOURNAL = {Apl. Mat.},
	FJOURNAL = {\v{C}eskoslovensk\'{a} Akademie V\v{e}d. Aplikace Matematiky},
	VOLUME = {29},
	YEAR = {1984},
	NUMBER = {1},
	PAGES = {23--39},
	ISSN = {0373-6725},
	MRCLASS = {34G20 (65M20)},
	MRNUMBER = {729950},
	MRREVIEWER = {Nicolae\ H.\ Pavel},
}

@article {SobolevskiiChebotareva1977,
	AUTHOR = {Sobolevski\u{\i}, P. E. and \v{C}ebotareva, L. M.},
	TITLE = {Approximate solution of the {C}auchy problem for an abstract
	hyperbolic equation by the method of lines},
	JOURNAL = {Izv. Vys\v{s}. U\v{c}ebn. Zaved. Matematika},
	FJOURNAL = {Izvestija Vys\v{s}ih U\v{c}ebnyh Zavedeni\u{\i} Matematika},
	YEAR = {1977},
	NUMBER = {5(180)},
	PAGES = {103--116},
	ISSN = {0021-3446},
	MRCLASS = {65M20},
	MRNUMBER = {520156},
	MRREVIEWER = {Russell\ C.\ Thompson},
}

@incollection {Arendt2004,
	AUTHOR = {Arendt, W.},
	TITLE = {Semigroups and evolution equations: functional calculus,
	regularity and kernel estimates},
	BOOKTITLE = {Evolutionary equations. {V}ol. {I}},
	SERIES = {Handb. Differ. Equ.},
	PAGES = {1--85},
	PUBLISHER = {North-Holland, Amsterdam},
	YEAR = {2004},
	ISBN = {0-444-51131-8},
	MRCLASS = {47D06 (34G10 34G20 35B65 35C15 35K90 47A60 47N20)},
	MRNUMBER = {2103696},
	MRREVIEWER = {Atsushi\ Yagi},
}

@book {ArendtVogtVoigt2026,
	AUTHOR = {Arendt, W. and Vogt, H. and Voigt, J.},
	TITLE = {Form methods for evolution equations},
	SERIES = {Operator Theory: Advances and Applications},
	VOLUME = {316},
	PUBLISHER = {Birkh{\"a}user/Springer, Cham},
	YEAR = {[2026] \copyright 2026},
	PAGES = {x+360},
	ISBN = {978-3-032-28410-5; 978-3-032-28411-2},
}

@book {HillePhillips1974,
	AUTHOR = {Hille, E. and Phillips, R. S.},
	TITLE = {Functional analysis and semi-groups},
	SERIES = {American Mathematical Society Colloquium Publications, Vol.
	XXXI},
	NOTE = {Third printing of the revised edition of 1957},
	PUBLISHER = {American Mathematical Society, Providence, RI},
	YEAR = {1974},
	PAGES = {xii+808},
	MRCLASS = {47-02 (46-02)},
	MRNUMBER = {423094},
}

@book {ZagrebnovNeidhardtIchinose2024,
	AUTHOR = {Zagrebnov, V. A. and Neidhardt, H. and Ichinose, T.},
	TITLE = {Trotter-{K}ato product formul\ae },
	SERIES = {Operator Theory: Advances and Applications},
	VOLUME = {296},
	PUBLISHER = {Birkh\"{a}user/Springer, Cham},
	YEAR = {[2024] \copyright 2024},
	PAGES = {xx+873},
	ISBN = {978-3-031-56719-3; 978-3-031-56720-9},
	MRCLASS = {47-02 (35K90 47A58 47A60 47B25 47Dxx)},
	MRNUMBER = {4763438},
	MRREVIEWER = {Kiyoko\ Furuya},
	DOI = {10.1007/978-3-031-56720-9},
	URL = {https://doi.org/10.1007/978-3-031-56720-9},
}

@article {GegechkoriRogavaTsiklauri2002,
	AUTHOR = {Gegechkori, Z. and Rogava, J. and Tsiklauri, M.},
	TITLE = {High degree precision decomposition method for the evolution
	problem with an operator under a split form},
	JOURNAL = {M2AN Math. Model. Numer. Anal.},
	FJOURNAL = {M2AN. Mathematical Modelling and Numerical Analysis},
	VOLUME = {36},
	YEAR = {2002},
	NUMBER = {4},
	PAGES = {693--704},
	ISSN = {0764-583X,1290-3841},
	MRCLASS = {34G10 (65M12 65M15)},
	MRNUMBER = {1932309},
	MRREVIEWER = {Cesar\ Palencia de Lara},
	DOI = {10.1051/m2an:2002030},
	URL = {https://doi.org/10.1051/m2an:2002030},
}

@article {GegechkoriRogavaTsiklauri2004,
	AUTHOR = {Gegechkori, Z. and Rogava, J. and Tsiklauri, M.},
	TITLE = {The fourth order accuracy decomposition scheme for an
	evolution problem},
	JOURNAL = {M2AN Math. Model. Numer. Anal.},
	FJOURNAL = {M2AN. Mathematical Modelling and Numerical Analysis},
	VOLUME = {38},
	YEAR = {2004},
	NUMBER = {4},
	PAGES = {707--722},
	ISSN = {0764-583X,1290-3841},
	MRCLASS = {65J05 (47D03)},
	MRNUMBER = {2087731},
	MRREVIEWER = {Jacques\ Rappaz},
	DOI = {10.1051/m2an:2004031},
	URL = {https://doi.org/10.1051/m2an:2004031},
}

@article {LomtadzeRogavaTsiklauri2004,
	AUTHOR = {Lomtadze, T. and Rogava, J. and Tsiklauri, M.},
	TITLE = {Approximate solution of {C}auchy problem for abstract
	hyperbolic equation using unitary group approximation method},
	JOURNAL = {Proc. I. Vekua Inst. Appl. Math.},
	FJOURNAL = {Proceedings of Ilia Vekua Institute of Applied Mathematics},
	VOLUME = {54/55},
	YEAR = {2004/05},
	PAGES = {55--64, 94},
	ISSN = {1512-004X},
	MRCLASS = {47D06 (65M12)},
	MRNUMBER = {2599419},
	MRREVIEWER = {Hassan\ Emamirad},
}

@article {Serbin1979,
	AUTHOR = {Serbin, S. M.},
	TITLE = {Rational approximations of trigonometric matrices with
	application to second-order systems of differential equations},
	JOURNAL = {Appl. Math. Comput.},
	FJOURNAL = {Applied Mathematics and Computation},
	VOLUME = {5},
	YEAR = {1979},
	NUMBER = {1},
	PAGES = {75--92},
	ISSN = {0096-3003,1873-5649},
	MRCLASS = {65L05 (65D25)},
	MRNUMBER = {516304},
	DOI = {10.1016/0096-3003(79)90011-0},
	URL = {https://doi.org/10.1016/0096-3003(79)90011-0},
}

@article {AlonsoMalloCanoMoreta2009,
	AUTHOR = {Alonso-Mallo, I. and Cano, B. and Moreta, M. J.},
	TITLE = {The stability of rational approximations of cosine functions
	on {H}ilbert spaces},
	JOURNAL = {Appl. Numer. Math.},
	FJOURNAL = {Applied Numerical Mathematics. An IMACS Journal},
	VOLUME = {59},
	YEAR = {2009},
	NUMBER = {1},
	PAGES = {21--38},
	ISSN = {0168-9274,1873-5460},
	MRCLASS = {65L06 (47D06 47N20)},
	MRNUMBER = {2474101},
	MRREVIEWER = {Lucas\ J\'{o}dar},
	DOI = {10.1016/j.apnum.2007.11.022},
	URL = {https://doi.org/10.1016/j.apnum.2007.11.022},
}

@article {GrimmHochbruck2008,
	AUTHOR = {Grimm, V. and Hochbruck, M.},
	TITLE = {Rational approximation to trigonometric operators},
	JOURNAL = {BIT},
	FJOURNAL = {BIT. Numerical Mathematics},
	VOLUME = {48},
	YEAR = {2008},
	NUMBER = {2},
	PAGES = {215--229},
	ISSN = {0006-3835,1572-9125},
	MRCLASS = {65F10},
	MRNUMBER = {2430617},
	MRREVIEWER = {Valeria\ Ruggiero},
	DOI = {10.1007/s10543-008-0185-9},
	URL = {https://doi.org/10.1007/s10543-008-0185-9},
}

@article {CanoMoreta2010,
	AUTHOR = {Cano, B. and Moreta, M. J.},
	TITLE = {Multistep cosine methods for second-order partial differential
	systems},
	JOURNAL = {IMA J. Numer. Anal.},
	FJOURNAL = {IMA Journal of Numerical Analysis},
	VOLUME = {30},
	YEAR = {2010},
	NUMBER = {2},
	PAGES = {431--461},
	ISSN = {0272-4979,1464-3642},
	MRCLASS = {65M99 (35B10 35L70)},
	MRNUMBER = {2608467},
	MRREVIEWER = {R\'{e}mi\ Vaillancourt},
	DOI = {10.1093/imanum/drn043},
	URL = {https://doi.org/10.1093/imanum/drn043},
}

@article {Gauckler2015,
	AUTHOR = {Gauckler, L.},
	TITLE = {Error analysis of trigonometric integrators for semilinear
	wave equations},
	JOURNAL = {SIAM J. Numer. Anal.},
	FJOURNAL = {SIAM Journal on Numerical Analysis},
	VOLUME = {53},
	YEAR = {2015},
	NUMBER = {2},
	PAGES = {1082--1106},
	ISSN = {0036-1429,1095-7170},
	MRCLASS = {65M20 (65L70 65M15 65P10)},
	MRNUMBER = {3335499},
	MRREVIEWER = {B\"{u}lent\ Karas\"{o}zen},
	DOI = {10.1137/140977217},
	URL = {https://doi.org/10.1137/140977217},
}

@article {BuchholzDoerichHochbruck2021,
	AUTHOR = {Buchholz, S. and D\"{o}rich, B. and Hochbruck, M.},
	TITLE = {On averaged exponential integrators for semilinear wave
	equations with solutions of low-regularity},
	JOURNAL = {Partial Differ. Equ. Appl.},
	FJOURNAL = {Partial Differential Equations and Applications},
	VOLUME = {2},
	YEAR = {2021},
	NUMBER = {2},
	PAGES = {Paper No. 23, 27},
	ISSN = {2662-2963,2662-2971},
	MRCLASS = {65J08 (35L05 65M12 65M15)},
	MRNUMBER = {4298262},
	MRREVIEWER = {Christian\ Budde},
	DOI = {10.1007/s42985-020-00045-9},
	URL = {https://doi.org/10.1007/s42985-020-00045-9},
}
		
		% === METHOD: (Optional) Force-add a PDF bookmark ============================
		% Use ONLY if your class/template fails to produce the bookmark above.
		% Toggle the switch below to true when you need the manual bookmark.
		\newif\ifForceBibBookmark
		% \ForceBibBookmarkfalse    % default: off (prevents duplicate bookmarks)
		\ForceBibBookmarktrue     % turn on if class/template does not create bookmark
		
		\ifForceBibBookmark
		\phantomsection                  % uncomment if you need to adjust the anchor position
		\pdfbookmark[1]{References}{bib} % level-1 bookmark titled "References", anchor name "bib"
		\fi
		
	\end{document}